\documentclass[11pt]{article}%
\usepackage{amsmath}
\usepackage{amsfonts}
\usepackage{amssymb}
\usepackage{graphicx}%
\newtheorem{theorem}{Theorem}

\newtheorem{definition}[theorem]{Definition}

\newtheorem{proposition}[theorem]{Proposition}
\newtheorem{remark}[theorem]{Remark}

\newenvironment{proof}[1][Proof]{\noindent\textbf{#1.} }{\ \rule{0.5em}{0.5em}}
\begin{document}

\title{Singular nonsymmetric Jack polynomials for some rectangular tableaux}
\author{Charles F. Dunkl\\Department of Mathematics\\University of Virginia\\Charlottesville, VA 22904-4137, USA}
\date{2 March 2020}
\maketitle

\begin{abstract}
In the intersection of the theories of nonsymmetric Jack polynomials in $N$
variables and representations of the symmetric groups $\mathcal{S}_{N}$ one
finds the singular polynomials. For certain values of the parameter $\kappa$
there are Jack polynomials which span an irreducible $\mathcal{S}_{N}$-module
and are annihilated by the Dunkl operators. The $\mathcal{S}_{N}$-module is
labeled by a partition of $N$, called the isotype of the polynomials. In this
paper the Jack polynomials are of the vector-valued type, that is, elements of
the tensor product of the scalar polynomials with the span of reverse standard
Young tableaux of the shape of a fixed partition of $N$. In particular this
partition is of shape $\left(  m,m,\ldots,m\right)  $ with $2k$ components and
the constructed singular polynomials are of isotype $\left(  mk,mk\right)  $
for the parameter $\kappa=$ $1/\left(  m+2\right)  $. The paper contains the
necessary background on nonsymmetric Jack polynomials and representation
theory and explains the role of Jucys-Murphy elements in the construction. The
main ingredient is the proof of uniqueness of certain spectral vectors,
namely, the list of eigenvalues of the Jack polynomials for the
Cherednik-Dunkl operators, when specialized to $\kappa=1/\left(  m+2\right)
$. The paper finishes with a discussion of associated maps of modules of the
rational Cherednik algebra and an example illustrating the difficulty of
finding singular polynomials for arbitrary partitions.

\end{abstract}

\section{Introduction}

In the study of polynomials in several variables there are two approaches, one
is algebraic which may involve symmetry groups generated by permutations of
coordinates and sign changes, for example, and the analytic approach includes
orthogonality with respect to weight functions and related calculus. The two
concepts are combined in the theory of Dunkl operators, which form a
commutative algebra of differential-difference operators, determined by a
reflection group $G$ and a parameter, and which are an analog of partial
derivatives. The relevant weight functions are products of powers of linear
functions vanishing on the mirrors and which are invariant under the
reflection group $G$. In the particular case of the objects of our study,
namely the symmetric groups $\mathcal{S}_{N}$, an orthogonal basis of
polynomials (called \textit{nonsymmetric Jack polynomials}) is constructed as
the set of simultaneous eigenfunctions of the Cherednik-Dunkl operators. This
is a commutative set of operators, self-adjoint for an inner product related
to the weight function. The inner product is positive-definite for an interval
of parameter values but for a discrete set of values there exist null
polynomials (that is, $\left\langle p,p\right\rangle =0$). It is these
parameter values that concern us here. The set of such polynomials of minimal
degree has interesting algebraic structure: in general it is a linear space
and an irreducible module of $\mathcal{S}_{N}$. The theory for scalar
polynomials is by now well understood \cite{D2004}, and the open problems
concern \textit{vector-valued} polynomials whose values lie in irreducible
modules. That is, the symmetric group $\mathcal{S}_{N}$ acts not only on the
domain but also the range of the polynomials. The key device for dealing with
the representation theory is to analyze when a polynomial is a simultaneous
eigenfunction of the Cherednik-Dunkl operators and of the Jucys-Murphy
elements with the same respective eigenvalues. In Etingof and Stoica
\cite{ES2009} there is an analysis of the vanishing properties, that is, the
zero sets, of singular polynomials of the groups $\mathcal{S}_{N}$ as well as
results on singular polynomials associated with minimal values of the
parameter for general modules of $\mathcal{S}_{N}$ and for the exterior powers
of the reflection representation of any finite reflection group $G$ (see also
\cite{D2018}). Their methods do not involve Jack polynomials. Feigin and
Silantyev \cite{FS2012} found explicit formulas for all singular polynomials
which span a module isomorphic to the reflection representation of $G$.

This paper concerns polynomials taking values in the representation of the
symmetric group corresponding to a rectangular partition. In particular for
$\tau=\left(  m^{2k}\right)  $ (the superscript indicates multiplicity) we
construct nonsymmetric Jack polynomials in $2mk$ variables which are singular
(annihilated by the Dunkl operators) for the parameter $\frac{1}{m+2}$ and
which span a module isomorphic to the representation $\sigma=\left(
mk,mk\right)  $.

In Section \ref{Bkgr} we present the basic definitions of operators,
combinatorial objects used in the representation theory of the symmetric
groups, and vector-valued nonsymmetric Jack polynomials. Subsection
\ref{Trans} is a concise treatment of the formulas for the transformations of
the Jack polynomials under the simple reflections; in a sense the whole paper
is about the effect of various transformations on these polynomials. Our
combinatorial arguments depend on \textit{bricks}, our term for the $2\times
m$ rectangles making up the partitions of concern; the properties of bricks
and the tableaux built out of them are covered in Section \ref{Brick}. The
Jucys-Murphy elements form a commutative subalgebra of the group algebra of
$\mathcal{S}_{N}$ and are a key part of the proof that certain Jack
polynomials are singular. The details are in Section \ref{AJM}. To show
singularity we establish the existence of the needed Jack polynomials when the
parameter is specialized to $\frac{1}{m+2}$ and the machinery for this is
developed in Sections \ref{RedThm} and \ref{Uniq}. The existence of a class of
singular polynomials leads to constructing maps of modules of the rational
Cherednik algebra; this topic is covered in Section \ref{Maps}. Finally
Section \ref{Conclude} discusses an easy generalization and also describes an
example which demonstrates the limits of the theory and introduces open problems.

\section{Background\label{Bkgr}}

The \textit{symmetric group} $\mathcal{S}_{N}$ acts on $\mathbb{R}^{N}$ by
permutation of coordinates. The space of polynomials is $\mathcal{P}%
:=\mathrm{s}\mathrm{p}\mathrm{a}\mathrm{n}_{\mathbb{R}\left(  \kappa\right)
}\left\{  x^{\alpha}:\alpha\in\mathbb{N}_{0}^{N}\right\}  $ where $\kappa$ is
a parameter and $\mathbb{N}_{0}=\left\{  0,1,2,3,\ldots\right\}  $. For
$\alpha\in\mathbb{N}_{0}^{N}$ set $\left\vert \alpha\right\vert =\sum
_{i=1}^{N}\alpha_{i}$. The action of $\mathcal{S}_{N}$ is extended to
polynomials by $wp\left(  x\right)  =p\left(  xw\right)  $ where $\left(
xw\right)  _{i}=x_{w\left(  i\right)  }$ (consider $x$ as a row vector and $w$
as a permutation matrix, $\left[  w\right]  _{ij}=\delta_{i,w\left(  j\right)
}$, then $xw=x\left[  w\right]  $). This is a representation of $\mathcal{S}%
_{N}$, that is, $w_{1}\left(  w_{2}p\right)  \left(  x\right)  =\left(
w_{2}p\right)  \left(  xw_{1}\right)  =p\left(  xw_{1}w_{2}\right)  =\left(
w_{1}w_{2}\right)  p\left(  x\right)  $ for all $w_{1},w_{2}\in\mathcal{S}%
_{N}$. Our structures depend on a transcendental (formal) parameter $\kappa$,
which may be specialized to a specific rational value $\kappa_{0}$.

Furthermore $\mathcal{S}_{N}$ is generated by reflections in the mirrors
$\left\{  x:x_{i}=x_{j}\right\}  $ for $1\leq i<j\leq N$. These are
\textit{transpositions, }denoted by $\left(  i,j\right)  $, so that $x\left(
i,j\right)  $ denotes the result of interchanging $x_{i}$ and $x_{j}$. Define
the $\mathcal{S}_{N}$-action on $\alpha\in\mathbb{Z}^{N}$ so that $\left(
xw\right)  ^{\alpha}=x^{w\alpha}$%
\[
\left(  xw\right)  ^{\alpha}=\prod_{i=1}^{N}x_{w\left(  i\right)  }%
^{\alpha_{i}}=\prod_{j=1}^{N}x_{j}^{\alpha_{w^{-1}\left(  j\right)  }}\text{,}%
\]
that is $\left(  w\alpha\right)  _{i}=\alpha_{w^{-1}\left(  i\right)  }$ (so
$\alpha$ is taken as a column vector and $w\alpha=\left[  w\right]  \alpha$).

The \textit{simple reflections} $s_{i}:=\left(  i ,i +1\right)  $, $1 \leq i
\leq N -1$, generate $\mathcal{S}_{N}$. They are the key devices for applying
inductive methods, and satisfy the \textit{braid} relations:
\begin{align*}
s_{i} s_{j}  &  =s_{j} s_{i} ,\left\vert i -j\right\vert \geq2 ;\\
s_{i} s_{i +1} s_{i}  &  =s_{i +1} s_{i} s_{i +1}\text{.}%
\end{align*}

We consider the situation where the group $\mathcal{S}_{N}$ acts on the range
as well as on the domain of the polynomials. We use vector spaces, called
$\mathcal{S}_{N}$-modules, on which $\mathcal{S}_{N}$ has an irreducible
orthogonal representation$.$ See James and Kerber \cite{JK2009} for
representation theory, including a discussion of Young's methods.

Denote the set of \textit{partitions}%
\[
\mathbb{N}_{0}^{N,+}:=\left\{  \lambda\in\mathbb{N}_{0}^{N}:\lambda_{1}%
\geq\lambda_{2}\geq\cdots\geq\lambda_{N}\right\}  \text{.}%
\]
An irreducible representation $\tau$ of $\mathcal{S}_{N}$ corresponds to a
partition of $N$ given the same label, that is $\tau\in\mathbb{N}_{0}^{N,+}$
and $\left\vert \tau\right\vert =N$. The length of $\tau$ is $\ell\left(
\tau\right)  =\max\left\{  i:\tau_{i}>0\right\}  $. There is a Ferrers diagram
of shape $\tau$ (also given the same label), with boxes at points $\left(
i,j\right)  $ with $1\leq i\leq\ell\left(  \tau\right)  $ and $1\leq j\leq
\tau_{i}$. A \textit{tableau} of shape $\tau$ is a filling of the boxes with
numbers, and a \textit{reverse standard Young tableau} (RSYT) is a filling
with the numbers $\left\{  1,2,\ldots,N\right\}  $ so that the entries
decrease in each row and each column. The set of RSYT of shape $\tau$ is
denoted by $\mathcal{Y}\left(  \tau\right)  $ and the representation is
realized on $V_{\tau}=\mathrm{s}\mathrm{p}\mathrm{a}\mathrm{n}_{\mathbb{R}%
\left(  \kappa\right)  }\left\{  T:T\in\mathcal{Y}\left(  \tau\right)
\right\}  $. For $1\leq i\leq N$ and $T\in\mathcal{Y}\left(  \tau\right)  $
the entry $i$ is at coordinates $\left(  \mathrm{\operatorname{row}}\left(
i,T\right)  ,\mathrm{\operatorname{col}}\left(  i,T\right)  \right)  $ and the
\textit{content} is $c\left(  i,T\right)  :=\mathrm{\operatorname{col}}\left(
i,T\right)  -\mathrm{\operatorname{row}}\left(  i,T\right)  $. We use
$T\left[  a,b\right]  $ to denote the entry at $\left(  a,b\right)  $, so
$\operatorname{row}\left(  T\left[  a,b\right]  ,T\right)
=a,\operatorname{col}\left(  T\left[  a,b\right]  =b\right)  $. Each
$T\in\mathcal{Y}\left(  \tau\right)  $ is uniquely determined by its
\textit{content vector} $\left[  c\left(  i,T\right)  \right]  _{i=1}^{N}$. A
sketch of the construction of $\tau$ is given in Subsection \ref{ieqi1} and
Remark \ref{reptau}. We are concerned with $\mathcal{P}_{\tau}%
=\mathcal{P\otimes V}_{\tau}$, that is, the $\mathrm{s}\mathrm{p}%
\mathrm{a}\mathrm{n}_{\mathbb{R}\left(  \kappa\right)  }\left\{  x^{\alpha
}\otimes T:\alpha\in\mathbb{N}_{0}^{N},T\in\mathcal{Y}\left(  \tau\right)
\right\}  $which is equipped with the $\mathcal{S}_{N}$ action:%
\[
w\left(  x^{\alpha}\otimes T\right)  :=\left(  xw\right)  ^{\alpha}\otimes
\tau\left(  w\right)  T,~\alpha\in\mathbb{N}_{0}^{N},T\in\mathcal{Y}\left(
\tau\right)  \text{,}%
\]
extended by linearity to%
\[
wp\left(  x\right)  =\tau\left(  w\right)  p\left(  xw\right)  ,~p\in
\mathcal{P}_{\tau}\text{.}%
\]

\begin{definition}
The \textit{Dunkl} and \textit{Cherednik-Dunkl} operators are ($1 \leq i \leq
N ,p \in\mathcal{P} ,T \in\mathcal{Y} \left(  \tau\right)  $)
\begin{align*}
\mathcal{D}_{i} \left(  p \left(  x\right)  \otimes T\right)   &  :=\frac{
\partial p \left(  x\right)  }{ \partial x_{i}} \otimes T +\kappa\sum_{j \neq
i}\frac{p \left(  x\right)  -p \left(  x \left(  i ,j\right)  \right)  }{x_{i}
-x_{j}} \otimes\tau\left(  \left(  i ,j\right)  \right)  T\text{,}\\
\mathcal{U}_{i} \left(  p \left(  x\right)  \otimes T\right)   &
:=\mathcal{D}_{i} \left(  x_{i} p \left(  x\right)  \otimes T\right)
-\kappa\sum_{j =1}^{i -1}p \left(  x \left(  i ,j\right)  \right)  \otimes
\tau\left(  \left(  i ,j\right)  \right)  T\text{,}%
\end{align*}
extended by linearity to all of $\mathcal{P}_{\tau}$.
\end{definition}

The commutation relations analogous to the scalar case hold, that is,%

\begin{align*}
\mathcal{D}_{i}\mathcal{D}_{j}  &  =\mathcal{D}_{j}\mathcal{D}_{i}%
,~\mathcal{U}_{i}\mathcal{U}_{j}=\mathcal{U}_{j}\mathcal{U}_{i},~1\leq i,j\leq
N\\
w\mathcal{D}_{i}  &  =\mathcal{D}_{w\left(  i\right)  }w,\forall
w\in\mathcal{S}_{N};~s_{j}\mathcal{U}_{i}=\mathcal{U}_{i}s_{j},~j\neq i-1,i;\\
s_{i}\mathcal{U}_{i}s_{i}  &  =\mathcal{U}_{i+1}+\kappa s_{i},~\mathcal{U}%
_{i}s_{i}=s_{i}\mathcal{U}_{i+1}+\kappa,~\mathcal{U}_{i+1}s_{i}=s_{i}%
\mathcal{U}_{i}-\kappa\text{.}%
\end{align*}
The simultaneous eigenfunctions of $\left\{  \mathcal{U}_{i}\right\}  $ are
called (vector-valued) nonsymmetric Jack polynomials (NSJP). They are the type
$A$ special case of the polynomials constructed by Griffeth \cite{G2010} for
the complex reflection groups $G\left(  n.p.N\right)  $. For generic $\kappa$
these eigenfunctions form a basis of $\mathcal{P}_{\tau}$ (\textit{generic}
means that $\kappa\neq\frac{m}{n}$ where $m,n\in\mathbb{Z}$ and $1\leq n\leq
N$). They have a triangularity property with respect to the partial order
$\rhd$ on compositions, which is derived from the dominance order:
\begin{align*}
\alpha &  \prec\beta~\Longleftrightarrow\sum_{j=1}^{i}\alpha_{j}\leq\sum
_{j=1}^{i}\beta_{j},~1\leq i\leq N,~\alpha\neq\beta\text{,}\\
\alpha\lhd\beta &  \Longleftrightarrow\left(  \left\vert \alpha\right\vert
=\left\vert \beta\right\vert \right)  \wedge\left[  \left(  \alpha^{+}%
\prec\beta^{+}\right)  \vee\left(  \alpha^{+}=\beta^{+}\wedge\alpha\prec
\beta\right)  \right]  \text{.}%
\end{align*}
There is a subtlety in the leading terms, which relies on the \textit{rank}
function $r_{\alpha}$:

\begin{definition}
For $\alpha\in\mathbb{N}_{0}^{N} ,1 \leq i \leq N$%
\[
r_{\alpha} \left(  i\right)  =\# \left\{  j :\alpha_{j} >\alpha_{i}\right\}
+\# \left\{  j :1 \leq j \leq i ,\alpha_{j} =\alpha_{i}\right\}  \text{,}%
\]
then $r_{\alpha} \in\mathcal{S}_{N}\text{.}$
\end{definition}

A consequence is that $r_{\alpha}\alpha=\alpha^{+}$, the \textit{nonincreasing
rearrangement} of $\alpha$, for any $\alpha\in\mathbb{N}_{0}^{N}$ . For
example if $\alpha=\left(  1,2,1,5,4\right)  $ then $r_{\alpha}=\left[
4,3,5,1,2\right]  $ and $r_{\alpha}\alpha=\alpha^{+}=\left(  5,4,2,1,1\right)
$ (recall $w\alpha_{i}=\alpha_{w^{-1}\left(  i\right)  }$ ). Also $r_{\alpha
}=I$ if and only if $\alpha\in\mathbb{N}_{0}^{N,+}$.

For each $\alpha\in\mathbb{N}_{0}^{N}$ and $T\in\mathcal{Y}\left(
\tau\right)  $ there is a NSJP $J_{\alpha,T}$ with leading term $x^{\alpha
}\otimes\tau\left(  r_{\alpha}^{-1}\right)  T$, that is,
\begin{equation}
J_{\alpha,T}\left(  x\right)  =x^{\alpha}\otimes\tau\left(  r_{\alpha}%
^{-1}\right)  T+\sum_{\alpha\rhd\beta}x^{\beta}\otimes v_{\alpha,\beta
,T}\left(  \kappa\right)  \label{dNJP}%
\end{equation}
where $v_{\alpha,\beta,S}\left(  \kappa\right)  \in V_{\tau}$; the
coefficients are rational functions of $\kappa$. These polynomials satisfy%
\begin{align*}
\mathcal{U}_{i}J_{\alpha,S} &  =\zeta_{\alpha,S}\left(  i\right)  J_{\alpha
,S},\\
\zeta_{\alpha,S}\left(  i\right)   &  :=\alpha_{i}+1+\kappa c\left(
r_{\alpha}\left(  i\right)  ,S\right)  ,~1\leq i\leq N.
\end{align*}
For detailed proofs see \cite{DL2011}. The commutation
\[
\mathcal{D}_{i}x_{i}=x_{i}\mathcal{D}_{i}+1+\kappa\sum_{j\neq i}\left(
i,j\right)
\]
implies%
\begin{equation}
\mathcal{U}_{i}=x_{i}\mathcal{D}_{i}+1+\kappa\sum_{j>i}\left(  i,j\right)
.\label{U2JM}%
\end{equation}
This introduces the definition of \textit{Jucys-Murphy elements} in the group
algebra $\mathbb{R}\mathcal{S}_{N}$
\[
\omega_{i}=\sum_{j=i+1}^{N}\left(  i,j\right)  ,~1\leq i<N;~\omega_{N}=0;
\]
which satisfy
\begin{align*}
\omega_{i}\omega_{j} &  =\omega_{j}\omega_{i},\\
s_{i}\omega_{j} &  =\omega_{j}s_{i},~\left\vert i-j\right\vert \geq2,\\
s_{i}\omega_{i}s_{i} &  =\omega_{i+1}+s_{i}.
\end{align*}
They act on $V_{\tau}$ by $\tau\left(  \omega_{i}\right)  T=\sum_{j>i}%
\tau\left(  \left(  i,j\right)  \right)  T=c\left(  i,T\right)  T.$ We will
use the modified operators $\mathcal{U}_{i}^{\prime}=\frac{1}{\kappa}\left(
\mathcal{U}_{i}-1\right)  =\frac{1}{\kappa}x_{i}\mathcal{D}_{i}+\omega_{i}$ .
The associated spectral vector is
\[
\zeta_{\alpha,t}^{\prime}\left(  i\right)  :=\frac{\alpha_{i}}{\kappa
}+c\left(  r_{\alpha}\left(  i\right)  ,T\right)  ,
\]
so that $\mathcal{U}_{i}^{\prime}J_{\alpha,T}=\zeta_{\alpha,t}^{\prime}\left(
i\right)  J_{\alpha,T}$ for $1\leq i\leq N$.

Throughout we use the phrase \textquotedblleft at $\kappa=\kappa
_{0}\textquotedblright$ where $\kappa_{0}$ is a rational number to mean that
the operators $\mathcal{U}_{i}^{\prime}$ and polynomials $J_{\alpha,S}$ are
evaluated at $\kappa=\kappa_{0}$. The transformation formulas and eigenvalue
properties are polynomial in $x$ and rational in $\kappa$. Thus the various
relations hold provided there is no pole. Hence to validly specialize to
$\kappa=\kappa_{0}$ it is necessary to prove the absence of poles.

Suppose $p\in\mathcal{P}_{\tau}$ and $1\leq i\leq N$ then $\mathcal{D}_{i}p=0$
if and only if $\mathcal{U}_{i}^{\prime}p=\omega_{i}p$ at $\kappa=\kappa_{0}$
(obvious from (\ref{U2JM})). The polynomial $p$ is said to be
\textit{singular} and $\kappa_{0}$ is a \textit{singular value}. From the
representation theory of $\mathcal{S}_{N}$ it is known that an irreducible
$\mathcal{S}_{N}$-module is isomorphic to an abstract space whose basis
consists of RSYT's of shape $\sigma$, a partition of $N$. The eigenvalues of
$\left\{  \omega_{i}\right\}  $ form content vectors which uniquely define an
RSYT. Suppose $\sigma$ is a partition of $N$ then a basis $\left\{  p_{S}%
:S\in\mathcal{Y}\left(  \sigma\right)  \right\}  $ (of an $\mathcal{S}_{N}%
$-invariant subspace) is called \textit{a basis of isotype} $\sigma$ if each
$\omega_{i}p_{S}=c\left(  i,S\right)  p_{S}$ for $1\leq i\leq N$ and each
$p_{S}$. If some $p\in\mathcal{P}_{\tau}$ is a simultaneous eigenfunction of
$\left\{  \omega_{i}\right\}  $ with $\omega_{i}p=\gamma_{i}p$ for $1\leq
i\leq N$ then the representation theory of $\mathcal{S}_{N}$ implies that
$\left[  \gamma_{i}\right]  _{i=1}^{N}$ is the content vector of a uniquely
determined RSYT of shape $\sigma$ for some partition $\sigma$ of $N$; this
allows specifying the isotype of a single polynomial without referring to a
basis. The key point here is when a subspace does have a basis of isotype
$\sigma$ made up of NSJP's. specialized to a fixed rational $\kappa=\kappa
_{0}$.

In this paper we construct singular polynomials for the partition $\left(
m^{2k}\right)  $ of $N=2mk$ for the singular value $\kappa_{0}=\frac{1}{m+2}$
and which are of isotype $\sigma=\left(  mk,mk\right)  $ , with $m\geq
1,k\geq2$. To show that the nonsymmetric Jack polynomials in the construction
have no poles at $\kappa=\frac{1}{m+2}$ we use the devices of proving
uniqueness of spectral vectors and performing valid transformations of the
polynomials. The proof of singularity will follow once we show the relevant
polynomials are eigenfunctions of the Jucys-Murphy operators $\omega_{i}$.
These properties are proven by a sort of induction using the simple
reflections $s_{i}$. For this purpose we describe the effect of $s_{i}$ on
$J_{\alpha,T}$.

One key device is to consider the related tableaux as a union of $k$
rectangles of shape $2\times m$, which we call \textit{bricks.}

\subsection{Review of transformation formulas\label{Trans}}

We collect formulas for the action of $s_{i}$ on $J_{\alpha,T}$. They will be
expressed in terms of the spectral vector $\zeta_{\alpha,T}^{\prime}=\left[
\dfrac{\alpha_{i}}{\kappa}+c\left(  r_{\alpha}\left(  i\right)  ,T\right)
\right]  _{i=1}^{2mk}$ and (for $1\leq i<2mk$)%
\begin{align*}
b_{\alpha,T}\left(  i\right)   &  =\frac{1}{\zeta_{\alpha,T}^{\prime}\left(
i\right)  -\zeta_{\alpha,T}^{\prime}\left(  i+1\right)  }\\
&  =\frac{\kappa}{\alpha_{i}-\alpha_{i+1}+\kappa\left(  c\left(  r_{\alpha
}\left(  i\right)  ,T\right)  -c\left(  r_{\alpha}\left(  i+1\right)
,T\right)  \right)  }.
\end{align*}
The formulas are consequences of the commutation relationships: $s_{j}%
\mathcal{U}_{i}^{\prime}=\mathcal{U}_{i}^{\prime}s_{j}$ for $j<i-1$ and $j>i$;
$s_{i}\mathcal{U}_{i}^{\prime}s_{i}=\mathcal{U}_{i+1}^{\prime}+s_{i}$ for
$1\leq i<2mk$. Observe that the formulas manifest the equation $\left(
s_{i}+b_{\alpha,T}\left(  i\right)  \right)  \left(  s_{i}-b_{\alpha,T}\left(
i\right)  \right)  =1-b_{\alpha,T}\left(  i\right)  ^{2}$.

\subsubsection{$\alpha_{i+1}>\alpha_{i}$\label{i1gti}$:$}%

\begin{align*}
\left(  s_{i}-b_{\alpha,T}\left(  i\right)  \right)  J_{\alpha,T}  &
=J_{s_{i}\alpha,T}\\
\left(  s_{i}+b_{\alpha,T}\left(  i\right)  \right)  J_{s_{i}\alpha,T}  &
=\left(  1-b_{\alpha,T}\left(  i\right)  ^{2}\right)  J_{\alpha,T}.
\end{align*}

\subsubsection{$\alpha_{i}>\alpha_{i+1}$\label{igti1}:}%

\begin{align*}
\left(  s_{i}-b_{\alpha,T}\left(  i\right)  \right)  J_{\alpha,T}  &  =\left(
1-b_{\alpha,T}\left(  i\right)  ^{2}\right)  J_{s_{i}\alpha,T},\\
\left(  s_{i}+b_{\alpha,T}\left(  i\right)  \right)  J_{s_{i}\alpha,T}  &
=J_{\alpha,T},
\end{align*}

\subsubsection{$\alpha_{i}=\alpha_{i+1},r_{\alpha}\left(  i\right)
=j$\label{ieqi1}:}

In this case $b_{\alpha,T}\left(  i\right)  =1/\left(  c\left(  j,T\right)
-c\left(  j+1,T\right)  \right)  $. Then if

\begin{enumerate}
\item $b_{\alpha,T}\left(  i\right)  =1,$ ($\operatorname{row}\left(
j,T\right)  =\operatorname{row}\left(  j+1,T\right)  $) $s_{i}J_{\alpha
,T}=J_{\alpha,T},$

\item $b_{\alpha,T}\left(  i\right)  =-1,$ $\left(  \operatorname{col}\left(
j,T\right)  =\operatorname{col}\left(  j+1,T\right)  \right)  $ $s_{i}%
J_{\alpha,T}=-J_{\alpha,T}$,

\item $0<b_{\alpha,T}\left(  i\right)  \leq\frac{1}{2}$ ($\operatorname{col}%
\left(  j,T\right)  >\operatorname{col}\left(  j+1,T\right)
,\operatorname{row}\left(  j,T\right)  <\operatorname{row}(j+1,T$)%
\begin{align*}
\left(  s_{i}-b_{\alpha,T}\left(  i\right)  \right)  J_{\alpha,T}  &
=J_{\alpha,T^{\left(  j\right)  }}\text{,}\\
\left(  s_{i}+b_{\alpha,T}\left(  i\right)  \right)  J_{\alpha,T^{\left(
j\right)  }}  &  =\left(  1-b_{\alpha,T}\left(  i\right)  ^{2}\right)
J_{\alpha,T},
\end{align*}

\item $-\frac{1}{2}\leq b_{\alpha,T}\left(  i\right)  <0$ $\left(
\operatorname{col}\left(  j,T\right)  <\operatorname{col}\left(  j+1,T\right)
,\operatorname{row}\left(  j,T\right)  >\operatorname{row}(j+1,T\right)  $%
\begin{align*}
\left(  s_{i}-b_{\alpha,T}\left(  i\right)  \right)  J_{\alpha,T}  &  =\left(
1-b_{\alpha,T}\left(  i\right)  ^{2}\right)  J_{\alpha,T^{\left(  j\right)  }%
},\\
\left(  s_{i}+b_{\alpha,T}\left(  i\right)  \right)  J_{\alpha,T^{\left(
j\right)  }}  &  =J_{\alpha,T}.
\end{align*}

\end{enumerate}

\begin{remark}
\label{reptau}The previous four formulas when restricted to $\left\{  1\otimes
T:T\in\mathcal{Y}\left(  \tau\right)  \right\}  $ (so that $\alpha=\left(
0,0,\ldots\right)  ,$ $r_{\alpha}\left(  i\right)  =i$ and $J_{\alpha
,T}=1\otimes T$) describe the action of $\tau$ on $V_{\tau}$; in this
situation $\mathcal{U}_{i}^{\prime}\left(  1\otimes T\right)  =c\left(
i,T\right)  \left(  1\otimes T\right)  $ for $1\leq i\leq2mk$.
\end{remark}

There is an important implication when $\alpha_{i}>\alpha_{i+1}$ and
$b_{\alpha,T}\left(  i\right)  =\pm1$ (at $\kappa=\kappa_{0}$) the general
relation $\left(  s_{i}-b_{\alpha,T}\left(  i\right)  \right)  J_{\alpha
,T}=\left(  1-b_{\alpha,T}\left(  i\right)  ^{2}\right)  J_{s_{i}\alpha,T}$
becomes $s_{i}J_{\alpha,T}=b_{\alpha,T}\left(  i\right)  J_{\alpha,T}$
provided that $J_{s_{i}\alpha,T}$ does not have a pole. Our device for proving
this is to show uniqueness of the spectral vector of $J_{s_{i}\alpha,T}$ or of
another polynomial $J_{\beta,S}$ which can be transformed to $J_{s_{i}%
\alpha,T}$ by a sequence of invertible ($\left\vert \zeta_{\gamma,S}\left(
i\right)  -\zeta_{\gamma,S}^{\prime}\left(  i+1\right)  \right\vert \geq2$)
steps using the simple reflections $\left\{  s_{i}\right\}  $.

\section{Properties of bricks\label{Brick}}

A brick is a $2\times m$ tableau which is one of the $k$ congruent rectangles
making up the Ferrers diagram of $\sigma$ or $\tau$. Since it is clear from
the context we can use the same name for the appearance in $\sigma$ or $\tau$.
Let $0\leq\ell\leq k-1$, then $B_{\ell}$ is the part $\left\{  \left[
i,j\right]  :i=1,2,m\ell<j\leq\left(  m+1\right)  \ell\right\}  $ of $\sigma$
or the part $\left\{  \left[  i,j\right]  :i=2\ell+1,2\ell+2,1\leq j\leq
m\right\}  $ in $\tau$. The \textit{standard brick }$\widetilde{B}_{\ell}$ has
the entries entered column by column:%
\[
\widetilde{B}_{\ell}=%
\begin{bmatrix}
2m\left(  k-\ell\right)  & \cdots & 2m\left(  k-\ell-1\right)  +2\\
2m\left(  k-\ell\right)  -1 & \cdots & 2m\left(  k-\ell-1\right)  +1
\end{bmatrix}
.
\]
In this section we use bricks to construct for each $S\in\mathcal{Y}\left(
\sigma\right)  $ a pair $\left(  \beta,T\right)  \in\mathbb{N}_{0}^{2mk}%
\times\mathcal{Y}\left(  \tau\right)  $ such that $\left(  m+2\right)
\beta_{i}+c\left(  r_{\beta}\left(  i\right)  ,T\right)  =c\left(  i,S\right)
$ for $1\leq i\leq2mk$. Later on we will prove uniqueness of $\left(
\beta,T\right)  $.

For the partition $\sigma$ we use the distinguished RSYT $S_{0}$ formed by
entering $2mk,2mk-1,\ldots,2,1$ column by column, that is $S_{0}$ is the
concatenation of $\widetilde{B}_{0}\widetilde{B}_{1}\cdots\widetilde{B}_{k-1}%
$. Observe $\#\mathcal{Y}\left(  \sigma\right)  =\frac{1}{mk+1}\binom{2mk}%
{mk}$, a Catalan number. The contents of $B_{\ell}$ in $\sigma$ are given by%
\[%
\begin{bmatrix}
m\ell & \cdots & m\left(  \ell+1\right)  -1\\
m\ell-1 & \cdots & m\left(  \ell+1\right)  -2
\end{bmatrix}
\]
Form the distinguished RSYT $T_{0}$ of shape $\tau$ by stacking the standard
bricks, from $\widetilde{B}_{0}$ at the top (rows $\#1$ and $\#2$) to
$\widetilde{B}_{k-1}$ at the bottom (rows $\#\left(  2k-1\right)  $ and
$\#\left(  2k\right)  $). The location of this brick in $T_{0}$ has corners
$\left[  2\ell+1,1\right]  $, $\left[  2\ell+1,m\right]  $, $~\left[
2\ell+2,1\right]  $, $\left[  2\ell+2,m\right]  $ (and the entries are
$2m,2m-1,\ldots,2,1$ entered column by column). Thus $T_{0}$ has the numbers
$2mk,2mk-1,\ldots,2,1$ entered column by column in each brick; here is the
example $m=3,k=2$%
\[
T_{0}=%
\begin{bmatrix}
12 & 10 & 8\\
11 & 9 & 7\\
6 & 4 & 2\\
5 & 3 & 1
\end{bmatrix}
.
\]
The contents for $B_{\ell}$ in $T_{0}$ are%
\[%
\begin{bmatrix}
-2\ell & \cdots & -2\ell+m-1\\
-2\ell-1 & \cdots & -2\ell+m-2
\end{bmatrix}
.
\]
Let $\lambda=\left(  \left(  k-1\right)  ^{2m},\left(  k-2\right)
^{2m},\ldots,1^{2m},0^{2m}\right)  $.

\begin{proposition}
The spectral vector at $\kappa=\frac{1}{m+2}$ of $J_{\lambda,T_{0}}$ equals
the content vector of $S_{0}$.
\end{proposition}

\begin{proof}
Since $\lambda\in\mathbb{N}_{0}^{2mk,+}$ if $i\in\widetilde{B}_{\ell}$ then
$\lambda_{i}=\ell$ and $\zeta_{\lambda,T_{0}}^{\prime}\left(  i\right)
=\left(  m+2\right)  \ell+c\left(  i,T_{0}\right)  $. By the structure of
$\widetilde{B}_{\ell}$ it suffices to check the value at one corner, say the
top left one. Taking $i=2m\left(  k-\ell\right)  $ with $c\left(
i,T_{0}\right)  =-2\ell$ we obtain $\zeta_{\lambda,T_{0}}^{\prime}\left(
i\right)  =\left(  m+2\right)  \ell-2\ell=m\ell$.
\end{proof}

Suppose $S\in\mathcal{Y}\left(  \sigma\right)  $ then there is a permutation
$\beta$ of $\lambda$ and an RSYT of shape $\tau$ such that $\zeta_{\beta
,T}^{\prime}\left(  i\right)  =c\left(  i,S\right)  $ at $\kappa=\frac{1}%
{m+2}$ for $1\leq i\leq2mk$. The construction is described in the following.

\begin{definition}
\label{S2BT}Suppose $S\in\mathcal{Y}\left(  \sigma\right)  $ and $0\leq
\ell\leq k-1$; and suppose the part of $S$ in the brick $B_{\ell}$ is%
\[%
\begin{bmatrix}
n_{1} & n_{2} & \cdots & \cdots & n_{m}\\
n_{m+1} & n_{m+2} & \cdots & \cdots & n_{2m}%
\end{bmatrix}
.
\]
The entries decrease in each column and in each row. Define $\beta
\in\mathbb{N}_{0}^{2mk}$ by $\beta_{n_{i}}=\ell$ for $1\leq i\leq2m$. Set up a
local rank function (for $1\leq i\leq2m$): $\rho_{i}=\#\left\{  j:n_{j}\geq
n_{i},1\leq j\leq2m\right\}  ,$ then $\rho_{2m}=1$ and $\rho_{1}=2m$.
Picturesquely, form a brick-shaped tableau by replacing $n_{i}$ by $\rho_{i}$
and then adding $2\ell m$ to each entry. Then stack these tableaux to form an
RSYT of shape $\tau$. Specifically set $T\left[  2\ell+1.i\right]  =\rho
_{i}+2\left(  k-1-\ell\right)  m,T\left[  2\ell+2,i\right]  =\rho
_{i+m}+2\left(  k-1-\ell\right)  m$ for $1\leq i\leq m$. Perform this
construction for each $\ell$ with $0\leq\ell\leq k-1$.
\end{definition}

Denote $\beta,T$ constructed in the Definition by $\beta\left\{  S\right\}
,T\left\{  S\right\}  $, or by the abbreviation $\pi\left\{  S\right\}  $.

\begin{proposition}
Suppose $S\in\mathcal{Y}\left(  \sigma\right)  $\emph{ }and $\beta
=\beta\left\{  S\right\}  ,T=T\left\{  S\right\}  $ then $\left(  m+2\right)
\beta_{i}+c\left(  r_{\beta}\left(  i,T\right)  \right)  =c\left(  i,S\right)
$ for all $i$.
\end{proposition}

\begin{proof}
Suppose $i$ is in brick $B_{\ell}$ and $i=n_{j}$ for some $j$. If $1\leq j\leq
m$ then $\operatorname{row}\left(  i,S\right)  =1$ and $c\left(  i,S\right)
=m\ell+j-1$, while if $m<j\leq2m$ then $\operatorname{row}\left(  i,S\right)
=2$ and $c\left(  i,S\right)  =m\ell+j-m-2$. Then \thinspace$\#\left\{
s:\beta_{s}=\beta_{i}=\ell,s\leq i\right\}  =\left\{  s:n_{s}\leq
n_{j}\right\}  =\rho_{j}$. Also $\#\left\{  s:\beta_{s}>\ell\right\}
=2m\left(  k-1-\ell\right)  $ and thus $r_{\beta}\left(  i\right)  =\rho
_{j}+2m\left(  k-1-\ell\right)  $.

If $1\leq j\leq m$ then $T\left[  2\ell+1,j\right]  =r_{\beta}\left(
i\right)  $, $c\left(  r_{\beta}\left(  i\right)  ,T\right)  =j-2\ell-1$ and
$\left(  m+2\right)  \beta_{i}+c\left(  r_{\beta}\left(  i,T\right)  \right)
\allowbreak=m\ell+j-1=c\left(  i,S\right)  $. If $m+1\leq j\leq2m$ then
$T\left[  2\ell+2,j-m\right]  =r_{\beta}\left(  i\right)  $, $c\left(
r_{\beta}\left(  i\right)  ,T\right)  =j-m-2\ell-2$ and $\left(  m+2\right)
\beta_{i}+c\left(  r_{\beta}\left(  i,T\right)  \right)  \allowbreak
=m\ell+\left(  j-m\right)  -2=c\left(  i,S\right)  $.
\end{proof}

Here is an example for $m=3,k=3,\kappa=\frac{1}{5}$%
\begin{align*}
S  &  =%
\begin{bmatrix}
18 & 17 & 13 & 14 & 10 & 8 & 7 & 6 & 3\\
16 & 15 & 12 & 11 & 9 & 5 & 4 & 2 & 1
\end{bmatrix}
,\\
\beta\left\{  S\right\}   &  =\left(
2,2,2,2,1,2,2,1,1,1,1,0,0,1,0,0,0,0\right)
\end{align*}
with the values of the local rank $\rho$%
\[%
\begin{bmatrix}
6 & 5 & 2\\
4 & 3 & 1
\end{bmatrix}
,%
\begin{bmatrix}
6 & 4 & 2\\
5 & 3 & 1
\end{bmatrix}
,%
\begin{bmatrix}
6 & 5 & 3\\
4 & 2 & 1
\end{bmatrix}
,
\]
now add $12,6,0$ respectively and combine to form%
\[
T\left\{  S\right\}  =%
\begin{bmatrix}
18 & 17 & 14\\
16 & 15 & 13\\
12 & 10 & 8\\
11 & 9 & 7\\
6 & 5 & 3\\
4 & 2 & 1
\end{bmatrix}
.
\]
For example $\beta_{10}=1,r_{\beta}\left(  10\right)  =10$ and $c\left(
10,T\right)  =-1$ thus $\zeta_{\beta,T}^{\prime}\left(  10\right)
=5-1=4=c\left(  10,S\right)  $.

Essentially what is left to do for the singularity proofs is to show\linebreak%
\ $\mathrm{span}\left\{  J_{\pi\left\{  S\right\}  }:S\in\mathcal{Y}\left(
\sigma\right)  \right\}  $ is closed under $\left\{  s_{i}:1\leq
i<2mk\right\}  $ and that $\omega_{i}J_{\pi\left\{  S\right\}  }=c\left(
i,S\right)  J_{\pi\left\{  S\right\}  }$ for all $i$. Here is a small example
of the impending difficulty: let $m=2,k=2$ $\left(  \kappa=\frac{1}{4}\right)
$ and%
\begin{align*}
S  &  =%
\begin{bmatrix}
8 & 6 & 5 & 2\\
7 & 4 & 3 & 1
\end{bmatrix}
,T=T\left\{  S\right\}  =%
\begin{bmatrix}
8 & 6\\
7 & 5\\
4 & 2\\
3 & 1
\end{bmatrix}
,\\
\beta &  =\beta\left\{  S\right\}  =\left(  1,1,1,0,1,0,0,0\right)  .
\end{align*}
What is the result of applying $s_{5}$? Interchanging $5$ and $6$ in $S$
results in a tableau violating the condition of decreasing entries in each row
(thus outside the span), and the general transformation formula (\ref{igti1})
($\beta_{5}>\beta_{6}$) says%
\[
\left(  s_{5}-b_{\beta,T}\left(  5\right)  \right)  J_{\beta,T}=\left(
1-b_{\beta,T}\left(  5\right)  ^{2}\right)  J_{s_{5}\beta,T}%
\]
with $s_{5}\beta=\left(  1,1,1,0,0,1,0,0\right)  $ and $b_{\beta,T}\left(
5\right)  ^{-1}=\frac{1}{\kappa}\left(  1-0\right)  +c\left(  4,T\right)
-c\left(  6,T\right)  =\frac{1}{\kappa}+\left(  -2-1\right)  $, thus
$b_{\beta,T}\left(  5\right)  =1$ at $\kappa=\frac{1}{4}$. To show that the
formula gives $s_{5}J_{\beta,T}=J_{\beta,T}$ it is necessary to show
$J_{s_{5}\beta,T}$ has no pole at $\kappa=\frac{1}{4}$. These proofs comprise
a large part of the sequel.

\section{Action of Jucys-Murphy elements\label{AJM}}

The Jucys-Murphy elements satisfy $s_{j}\omega_{i}=\omega_{i}s_{j}$ for $j\neq
i-1,i$ and $s_{i}\omega_{i}s_{i}=\omega_{i+1}+s_{i}$ for $i<2mk$.

Suppose there is a subset $\mathcal{Z}\subset\mathbb{N}_{0}^{2mk}%
\times\mathcal{Y}\left(  \tau\right)  $ with the properties (spectral vectors
at $\kappa=\frac{1}{m+2}$, recall $b_{\alpha,T}\left(  i\right)  =\left(
\zeta_{\alpha,T}^{\prime}\left(  i\right)  -\zeta_{\alpha,T}^{\prime}\left(
i+1\right)  \right)  ^{-1}$):

\begin{enumerate}
\item $\left(  \beta,T\right)  \in\mathcal{Z}$ and $\left\vert b_{\beta
,T}\left(  i\right)  \right\vert \leq\frac{1}{2}$ implies $\left(
s_{i}-b_{\beta,T}\left(  i\right)  \right)  J_{\beta,T}=\gamma J_{\beta
^{\prime},T^{\prime}}$ for some $\gamma\neq0$ and $\left(  \beta^{\prime
},T^{\prime}\right)  \in\mathcal{Z}$; also $\zeta_{\beta^{\prime},T^{\prime}%
}^{\prime}=s_{i}\zeta_{\beta,T}^{\prime}$;

\item $\left(  \beta,T\right)  \in\mathcal{Z}$ and $b_{\beta,T}\left(
i\right)  =\pm1$ implies $s_{i}J_{\beta,T}=b_{\beta,T}\left(  i\right)
J_{\beta,T}$;

\item $\left(  \beta,T\right)  \in\mathcal{Z}$ implies $\beta_{2mk}=0$ and
thus $\zeta_{\beta,T}^{\prime}\left(  2mk\right)  =0$.
\end{enumerate}

The following is a basic theorem on representations of $\mathcal{S}_{N}$ and
we sketch the proof.

\begin{theorem}
\label{J2S}If $\mathcal{Z}\subset\mathbb{N}_{0}^{2mk}\times\mathcal{Y}\left(
\tau\right)  $ satisfies these properties then $\left(  \beta,T\right)
\in\mathcal{Z}$ implies $\omega_{i}J_{\beta,T}=\zeta_{\beta,T}^{\prime}\left(
i\right)  J_{\beta,T}$ for $1\leq i\leq2mk$.
\end{theorem}

\begin{proof}
Arguing by induction suppose $\omega_{j}J_{\beta,T}=\zeta_{\beta,T}^{\prime
}\left(  j\right)  J_{\beta,T}$ for all $\left(  \beta,T\right)
\in\mathcal{Z}$ and $i<j\leq2mk$. The start $i=2mk-1$ is given in the
hypotheses. Let $\left(  \beta,T\right)  \in\mathcal{Z}$ and suppose that
$b_{\beta,T}\left(  i\right)  =\pm1$, then%
\begin{align*}
\omega_{i}J_{\beta,T}  &  =\left(  s_{i}\omega_{i+1}s_{i}+s_{i}\right)
J_{\beta,T}=\left\{  b_{\beta,T}\left(  i\right)  ^{2}\zeta_{\beta,T}^{\prime
}\left(  i+1\right)  +b_{\beta,T}\left(  i\right)  \right\}  J_{\beta,T}\\
&  =\left\{  \zeta_{\beta,T}^{\prime}\left(  i+1\right)  +b_{\beta,T}\left(
i\right)  \right\}  J_{\beta,T}=\zeta_{\beta,T}^{\prime}\left(  i\right)
J_{\beta,T}.
\end{align*}
Next suppose $\left\vert b_{\beta,T}\left(  i\right)  \right\vert \leq\frac
{1}{2}$ and set $p=\gamma J_{\beta^{\prime},T^{\prime}}=\left(  s_{i}%
-b_{\beta,T}\left(  i\right)  \right)  J_{\beta,T}$, thus $\omega_{i+1}%
p=\zeta_{\beta,T}^{\prime}\left(  i\right)  p$ (inductive hypothesis). Then%
\begin{align*}
\omega_{i}J_{\beta,T}  &  =\left(  s_{i}\omega_{i+1}+1\right)  \left(
p+b_{\beta,T}\left(  i\right)  J_{\beta,T}\right) \\
&  =\left(  \zeta_{\beta,T}^{\prime}\left(  i\right)  s_{i}+1\right)
p+b_{\beta,T}\left(  i\right)  \left(  \zeta_{\beta,T}^{\prime}\left(
i+1\right)  s_{i}+1\right)  J_{\beta,T}\\
&  =\left\{  \left(  \zeta_{\beta,T}^{\prime}\left(  i\right)  s_{i}+1\right)
\left(  s_{i}-b_{\beta,T}\left(  i\right)  \right)  +b_{\beta,T}\left(
i\right)  \left(  \zeta_{\beta,T}^{\prime}\left(  i+1\right)  s_{i}+1\right)
\right\}  J_{\beta,T}\\
&  =\left\{  1-\zeta_{\beta,T}^{\prime}\left(  i\right)  b_{\beta,T}\left(
i\right)  +b_{\beta,T}\left(  i\right)  \zeta_{\beta,T}^{\prime}\left(
i+1\right)  \right\}  s_{i}J_{\beta,T}+\zeta_{\beta,T}^{\prime}\left(
i\right)  J_{\beta,T}\\
&  =\zeta_{\beta,T}^{\prime}\left(  i\right)  J_{\beta,T}.
\end{align*}
This completes the induction.
\end{proof}

We want to show that $\left\{  \left(  \beta\left\{  S\right\}  ,T\left\{
S\right\}  \right)  :S\in\mathcal{Y}\left(  \sigma\right)  \right\}  $ (as in
Definition \ref{S2BT}) satisfies the hypotheses of Theorem \ref{J2S}. From the
construction it is clear that $\beta\left\{  S\right\}  _{2mk}=0$ because
$S\left[  1,1\right]  =2mk$ and this cell is in $B_{0}$. Fix $S\in
\mathcal{Y}\left(  \sigma\right)  $ and $i<2mk$. Abbreviate $\beta
=\beta\left\{  S\right\}  ,T=T\left\{  S\right\}  ,b=b_{\beta,T}\left(
i\right)  $. There are several cases:

\begin{enumerate}
\item $\left\vert c\left(  i,S\right)  -c\left(  i+1,S\right)  \right\vert
\geq2$ then $\left\vert b\right\vert \leq\frac{1}{2}$ and $\left(
s_{i}-b\right)  J_{\beta_{,T}}=\gamma J_{\beta^{\prime},T^{\prime}}$ where
$\gamma\neq0$ and $\zeta_{\beta,T}^{\prime}\left(  j\right)  =c\left(
j,S^{\left(  i\right)  }\right)  $ for all $j.$Specifically if $\beta_{i}%
\neq\beta_{i+1}$ implying that $i$ and $i+1$ are in different bricks then
$\beta^{\prime}=s_{i}\beta$ and $T^{\prime}=T$, while if $\beta_{i}%
=\beta_{i+1}=\ell$ then $i,i+1\in B_{\ell}$ and $T^{\prime}$ is formed from
$T$ by transforming the part of $T$ in $B_{\ell}$ interchanging $r_{\beta
}\left(  i\right)  $ and $r_{\beta}\left(  i\right)  +1.$

\item $c\left(  i,S\right)  -c\left(  i+1,S\right)  =-1$ ($\operatorname{col}%
\left(  i,S\right)  =\operatorname{col}\left(  i+1,S\right)  $) then by
construction $\beta_{i}=\beta_{i+1}=\ell$ and $i,i+1\in B_{\ell}$; suppose
that $i+1=n_{j}$ in the notation of Definition \ref{S2BT}. By hypothesis
$i=n_{j+m}$, $\rho_{j+m}=\rho_{j}-1$. Then $T\left[  2\ell+1.j\right]
=\rho_{j}+2\left(  k-1-\ell\right)  m$ and $T\left[  2\ell+2.j\right]
=T\left[  2\ell+1.j\right]  -1.$ This implies $\operatorname{col}\left(
r_{\beta}\left(  i\right)  ,T\right)  =\operatorname{col}\left(  r_{\beta
}\left(  i\right)  +1,T\right)  $.. By (\ref{ieqi1}) $s_{i}J_{\beta
,T}=-J_{\beta,T}$.

\item $c\left(  i,S\right)  -c\left(  i+1,S\right)  =1$ $\left(
\operatorname{row}\left(  i,S\right)  =\operatorname{row}\left(  i+1,S\right)
\right)  $ and $\beta_{i}=\beta_{i+1}=\ell$; then $i,i+1\in B_{\ell}$; using
Definition \ref{S2BT} $i=n_{j}$ and $i+1=n_{j-1}$ for some $j$ with $2\leq
j\leq m$ or $m+2\leq j\leq2m$, and $\rho_{j-1}=\rho_{j}+1.$ Thus $r_{\beta
}\left(  i\right)  =\rho_{j}+2\left(  k-1-\ell\right)  m$. In the first case
$T\left[  2\ell+1.j\right]  =T\left[  2\ell+1.j-1\right]  -1$ and in the
second case $T\left[  2\ell+2.j-m\right]  =T\left[  2\ell+2.j-1-m\right]  -1$
and thus $s_{i}J_{\beta,T}=J_{\beta,T}$. by (\ref{ieqi1}).

\item $c\left(  i,S\right)  -c\left(  i+1,S\right)  =1$ $\left(
\operatorname{row}\left(  i,S\right)  =\operatorname{row}\left(  i+1,S\right)
\right)  $ and $\beta_{i}>\beta_{i+1}$ then $i+1\in B_{\ell-1},i\in B_{\ell}$
(because the entries of $S$ are decreasing in each row). Thus $i+1$ is in
position $n_{m}$ or $n_{2m}$ of $B_{\ell-1}$ and $i$ is $n_{1}$ or $n_{m+1}$
respectively of $B_{\ell}$. The relevant transformation formula is in
\ref{igti1} : $\left(  s_{i}-b_{\beta,T}\left(  i\right)  \right)  J_{\beta
,T}=\left(  1-b_{\beta,T}\left(  i\right)  ^{2}\right)  J_{s_{i}\beta,T}$. To
allow $\kappa=\frac{1}{m+2}$ in this equation and conclude $\left(
s_{i}-b_{\beta,T}\left(  i\right)  \right)  J_{\beta,T}=0$ it is necessary to
show $J_{s_{i}\beta,T}$ has no poles there.
\end{enumerate}

To complete the proof that $\omega_{i}J_{\beta\left\{  S\right\}  ,T\left\{
S\right\}  }=c\left(  i,S\right)  J_{\beta\left\{  S\right\}  ,T\left\{
S\right\}  }$ for $1\leq i\leq2mk$ and $S\in\mathcal{Y}\left(  \sigma\right)
$ (at $\kappa=\frac{1}{m+2}$) we will show each $J_{\beta\left\{  S\right\}
,T\left\{  S\right\}  }$ and $J_{s_{i}\beta,T}$ (as described in (4) above)
has no poles at $\kappa=\frac{1}{m+2}$. In the next section we show that it
suffices to analyze $1+2\left(  k-1\right)  $ specific tableaux.

\section{Reduction theorems\label{RedThm}}

Suppose some $J_{\beta,T}$ has been shown to be defined at $\kappa=\frac
{1}{m+2}$ (no poles) and $\left\vert \zeta_{\beta,T}^{\prime}\left(  i\right)
-\zeta_{\beta,T}^{\prime}\left(  i+1\right)  \right\vert \geq2$ then
$J_{\beta^{\prime},T^{\prime}}$ where $\zeta_{\beta^{\prime},T^{\prime}%
}^{\prime}=s_{i}\zeta_{\beta,T}^{\prime}$ is also defined (recall $\left(
s_{i}-b_{\beta,T}\left(  i\right)  \right)  J_{\beta,T}$ is a nonzero multiple
of $J_{s_{i}\beta,T}$ if $\beta_{i}\neq\beta_{i+1}$ or of $J_{\beta,T^{\left(
j\right)  }}$ if $\beta_{i}=\beta_{i+1}$ and $j=r_{\beta}\left(  i\right)  $.)
and the process is invertible. In other words if $\zeta_{\beta,T}^{\prime}$ is
a valid spectral vector and $\left\vert \zeta_{\beta,T}^{\prime}\left(
i\right)  -\zeta_{\beta,T}^{\prime}\left(  i+1\right)  \right\vert \geq2$ then
$s_{i}\zeta_{\beta,T}^{\prime}$ is also a valid spectral vector
(\textit{valid}. means that there is a NSJP with that spectral vector and it
has no pole at $\kappa=\frac{1}{m+2}$).

We consider column-strict tableaux $S$ of shape $\sigma$ which are either RSYT
or $S$ differs by one row-wise transposition from being an RSYT. Their content
vectors are used in the argument. \textit{Column-strict} means that the
entries in each column are decreasing.

\begin{definition}
Suppose $1\leq n<m$ and $j=1,2$ then $\mathcal{R}_{j,n}$ is the set of
tableaux $S$ of shape $\sigma$ such that $S$ is column-strict and $S^{\prime}$
defined by $S^{\prime}\left[  j,n\right]  =S\left[  j,n+1\right]  ,S^{\prime
}\left[  j,n+1\right]  =S\left[  j,n\right]  $ and $S^{\prime}\left[
a,b\right]  =S\left[  a,b\right]  $ for $b\neq n,n+1$ is an RSYT.
\end{definition}

Suppose $\zeta_{\beta,T}^{\prime}\left(  u\right)  =c\left(  u,S\right)  $ for
$1\leq u\leq2mk$, $\operatorname{row}\left(  i,S\right)  =2,\operatorname{row}%
\left(  i+1,S\right)  =1$ and $\operatorname{col}\left(  i,S\right)
<\operatorname{col}\left(  i+1,S\right)  $ then $\zeta_{\beta,T}^{\prime
}\left(  i\right)  -\zeta_{\beta,T}^{\prime}\left(  i+1\right)  \leq-2$ and
$s_{i}\zeta_{\beta,T}^{\prime}$ is a spectral vector associated with
$S^{\left(  i\right)  }$. Call this a \textit{permissible step}. In fact the
inequality $\zeta_{\beta,T}^{\prime}\left(  i\right)  -\zeta_{\beta,T}%
^{\prime}\left(  i+1\right)  \leq-2$ is equivalent to the row and column
property just stated. If $S\in\mathcal{Y}\left(  \sigma\right)  $ then
$S^{\left(  i\right)  }\in\mathcal{Y}\left(  \sigma\right)  $ and if
$S\in\mathcal{R}_{j,n}$ then $S^{\left(  j\right)  }\in\mathcal{R}_{j,n}$.
(because any row or column orderings do not change). For counting permissible
steps we define
\begin{equation}
\mathrm{inv}\left(  S\right)  =\#\left\{  \left(  a,b\right)
:a<b,\operatorname{row}\left(  a,S\right)  <\operatorname{row}\left(
b,S\right)  \right\}  \label{invS}%
\end{equation}
A permissible step $S\rightarrow S^{\left(  i\right)  }$ adds $1$ to
$\mathrm{inv}\left(  S\right)  $. The reduction process aims to apply
permissible steps until a inv-maximal tableau is reached. In $\mathcal{Y}%
\left(  \sigma\right)  $ the inv-maximal element is $S_{0}$ and $\mathrm{inv}%
\left(  S_{0}\right)  =\binom{mk}{2}$ .

\begin{definition}
\label{defSjn}For $1\leq n<m$ and $j=1,2$ define a distinguished element
$S_{\left(  j,n\right)  }$ of $\mathcal{R}_{j,n}$ by $S_{\left(  j,n\right)
}\left[  1,i\right]  =2mk+2-2i\,$, $S_{\left(  j,n\right)  }\left[
2,i\right]  =2mk+1-2i$ for $i\neq n,n+1$ and for $a=1,2,~b=n,n+1$
\begin{align*}
\left\{  S_{\left(  1,n\right)  }\left[  a,b\right]  \right\}   &  =%
\begin{Bmatrix}
2mk-2n+1 & 2mk-2n+2\\
2mk-2n & 2mk-2n-1
\end{Bmatrix}
,\\
\left\{  S_{\left(  2,n\right)  }\left[  a,b\right]  \right\}   &  =%
\begin{Bmatrix}
2mk-2n+2 & 2mk-2n+1\\
2mk-2n-1 & 2mk-2n
\end{Bmatrix}
.
\end{align*}

\end{definition}

Then $\mathrm{inv}S_{\left(  j,n\right)  }=\binom{mk}{2}-1$. Here are two
examples with $mk=6$:%
\[
S_{\left(  1,3\right)  }=%
\begin{bmatrix}
12 & 10 & 7 & 8 & 4 & 2\\
11 & 9 & 6 & 5 & 3 & 1
\end{bmatrix}
,S_{\left(  2,2\right)  }=%
\begin{bmatrix}
12 & 10 & 9 & 6 & 4 & 2\\
11 & 7 & 8 & 5 & 3 & 1
\end{bmatrix}
.
\]
Any $S\in\mathcal{Y}\left(  \sigma\right)  $ can be transformed by a sequence
of permissible steps to $S_{0}$ (this is a basic fact in representation theory
but the explanation is useful to motivate the argument for $\mathcal{R}_{j,n}%
$), and any $S\in\mathcal{R}_{j,n}$ can be transformed in this way to
$S_{\left(  j,n\right)  }$. For convenience replace $mk$ by $N$ since only the
number of columns is relevant. Suppose $S\in\mathcal{Y}\left(  \sigma\right)
$ and by permissible steps has been transformed to $S^{\prime}$ with
$S^{\prime}\left[  1,i\right]  =2N+2-2i,S^{\prime}\left[  2,i\right]
=2N+1-2i$ for $i\leq r<N-1$ (the inductive argument starts with $r=0$). From
the definition of $\mathcal{Y}\left(  \sigma\right)  $ it follows that
$S^{\prime}\left[  1,i\right]  <S^{\prime}\left[  1,2r+1\right]  $ and
$S^{\prime}\left[  2,i\right]  <S^{\prime}\left[  2,r+1\right]  <S^{\prime
}\left[  1,r+1\right]  $ for all $i>r+1.$ This implies $S^{\prime}\left[
1,r+1\right]  =2N-2r$ and $S^{\prime}\left[  2,r+1\right]  =2N-2r-u$ with
$u\geq1$. Then the list of entries $\left[  S^{\prime}\left[  1,\ell\right]
\right]  _{\ell=r+1}^{r+u}$ equals $\left[  2N-2r,2N-2r-2,\ldots
,2N-2r-u+1\right]  $ Apply $s_{2N-2r-u},s_{2N-2r-u+1},\ldots,s_{2N-2r-2}$ in
this order (if $u=1$ then already done). Each one is a permissible step, with
$t$ in $\left[  2,r+1\right]  $ and $t+1$ in $\left[  2,n^{\prime}\right]  $
with $n^{\prime}>r+1$ . This produces $S^{\prime\prime}$ satisfying
$S^{\prime\prime}\left[  1,i\right]  =2N+2-2i,S^{\prime^{\prime}}\left[
2,i\right]  =2N+1-2i$ for $i\leq r+1<N$. The induction stops at $r+1=N-1$.

Suppose $S\in\mathcal{R}_{j,n}$ and by permissible steps has been transformed
to $S^{\prime}$ with $S^{\prime}\left[  1,i\right]  =2N+2-2i,S^{\prime}\left[
2,i\right]  =2N+1-2i$ for $i\leq r<n-1$ (the inductive argument starts with
$r=0$). From the definition of $\mathcal{R}_{j,n}$ it follows that $S^{\prime
}\left[  1,i\right]  <S^{\prime}\left[  1,2r+1\right]  $ and $S^{\prime
}\left[  2,i\right]  <S^{\prime}\left[  2,r+1\right]  <S^{\prime}\left[
1,r+1\right]  $ for all $i>r+1.$ This implies $S^{\prime}\left[  1,r+1\right]
=2N-2r$ and $S^{\prime}\left[  2,r+1\right]  =2N-2r-u$ with $u\geq1$. The
numbers $2N-2r-u+1,2N-2r-u+2,\ldots,2N-2r$ are in $\left\{  S^{\prime}\left[
1,\ell\right]  :r+1\leq\ell\leq r+u\right\}  $. As in the RSYT case apply
$s_{2N-2r-u},s_{2N-2r-u+1},\ldots,s_{2N-2r-2}$ in this order (if $u=1$ then
already done). It is possible that one pair of adjacent entries is out of
order (when $j=1$) but the argument is still valid. Here is a small example
with $N=4,n=2,j=1,r=0$.%
\[
S=%
\begin{bmatrix}
8 & 6 & 7 & 3\\
5 & 4 & 2 & 1
\end{bmatrix}
\overset{s_{5}}{\rightarrow}%
\begin{bmatrix}
8 & 5 & 7 & 3\\
6 & 4 & 2 & 1
\end{bmatrix}
\overset{s_{6}}{\rightarrow}%
\begin{bmatrix}
8 & 5 & 6 & 3\\
7 & 4 & 2 & 1
\end{bmatrix}
.
\]
The inductive process can be continued until $r=n-1$ and the result is a
tableau $S^{\prime}\in\mathcal{R}_{j,n}$ with the entries $2N-2n+3,\ldots,2N$
in the first $n-1$ columns. Thus the entries $1,2,\ldots,2N-2n+2$ are in the
remaining $N-n+1$ columns.

The next part of the process is to start from the last column and work
forward. Suppose $S^{\prime}$ by permissible steps has been transformed to
$S^{\prime\prime}$ with $S^{\prime\prime}\left[  1,N+1-i\right]
=2i,S^{\prime\prime}\left[  2,N+1-i\right]  =2i-1$ for $i\leq r<N-n-2$ (the
first step is with $r=0$). As before $S^{\prime\prime}\left[  1,i\right]
>S^{\prime\prime}\left[  1,N-r\right]  >S^{\prime\prime}\left[  2,N-r\right]
$ and $S^{\prime\prime}\left[  2,i\right]  >S^{\prime\prime}\left[
2,N-r\right]  $ for $i\leq N-r-1$. This implies $S^{\prime\prime}\left[
2,N-r\right]  =2r+1$ and $S^{\prime\prime}\left[  1,N-r\right]  =2r+1+u$ with
$u\geq1$. The numbers $2r+1,2r+2,\ldots,2r+u$ are in $\left\{  S^{\prime
\prime}\left[  2,\ell\right]  :N-r-u+1\leq\ell\leq N-r\right\}  $. This range
of cells has contents $N-r-u-1,N-r-u,\ldots,N-r-3$ (excluding $\left[
2,N-r\right]  $). Possibly one pair of adjacent entries is out of order (when
$j=2$). In terms of contents while $c\left(  2r+1+u,S^{\prime\prime}\right)
=N-r-1$ so the steps $s_{2r+u},s_{2r+u-1},\ldots,s_{2r+u}$ are permissible in
that order resulting in $S^{\prime\prime\prime}$ with $S^{\prime\prime\prime
}\left[  1,N-r\right]  =2r+2$. The process stops at $N-r=n+2$. The result is
$S^{\prime\prime\prime}\left[  1,i\right]  =2N+2-2i,S^{\prime\prime\prime
}\left[  2,i\right]  =2N+1-2i$ for $1\leq i<n$ and $n+1<i\leq N$. Thus the
entries in columns $n$ and $n+1$ are $2N-1-2n,\ldots,2N+2-2n$. The definition
of $\mathcal{R}_{j,n}$ forces the position of these entries:%
\[
\mathcal{R}_{1,n}:%
\begin{bmatrix}
2N-2n+1 & 2N-2n+2\\
2N-2n & 2N-2n-1
\end{bmatrix}
,\mathcal{R}_{2,n}:%
\begin{bmatrix}
2N-2n+2 & 2N-2n+1\\
2N-2n-1 & 2N-2n
\end{bmatrix}
.
\]
Thus we have shown that any $S\in\mathcal{R}_{j,n}$ can be transformed by
permissible steps to $S_{\left(  j,n\right)  }$ an inv-maximal tableau.

In the above example $r=0$ and $u=2$ and the action of $s_{2}$ suffices to
obtain the desired tableau:%
\[
S^{\prime\prime}=%
\begin{bmatrix}
8 & 5 & 6 & 3\\
7 & 4 & 2 & 1
\end{bmatrix}
\overset{s_{2}}{\rightarrow}%
\begin{bmatrix}
8 & 5 & 6 & 2\\
7 & 4 & 3 & 1
\end{bmatrix}
;
\]
no more permissible steps are possible.

In our applications $N=mk$ and $n=\ell m$ with $1\leq\ell\leq k-1$.

\section{Uniqueness theorems\label{Uniq}}

This section starts by showing how uniqueness of spectral vectors is used to
prove that specific Jack polynomials exist for some $\kappa=\kappa_{0}$, that
is, there are no poles there.

\begin{proposition}
Suppose $\left(  \beta,T\right)  \in\mathbb{N}_{0}^{N}\times\mathcal{Y}\left(
\tau\right)  $ has the property that $\left(  \gamma,T^{\prime}\right)
\in\mathbb{N}_{0}^{N}\times\mathcal{Y}\left(  \tau\right)  $, $\gamma
\trianglelefteq\beta$ and $\zeta_{\gamma,T^{\prime}}^{\prime}\left(  i\right)
=\zeta_{\beta,T}^{\prime}\left(  i\right)  $ for $1\leq i\leq N$ at
$\kappa=\kappa_{0}$ implies $\left(  \gamma,T^{\prime}\right)  =\left(
\beta,T\right)  $ then $J_{\beta,T}$ is defined at $\kappa=\kappa_{0}$, in the
sense that the generic expression for $J_{\beta,T}$ can be specialized to
$\kappa=\kappa_{0}$ without poles.
\end{proposition}

\begin{proof}
From the $\vartriangleright$-triangular nature of (\ref{dNJP}) it follows that
the inversion formulas are also triangular, in particular%
\[
x^{\beta}\otimes\tau\left(  r_{\beta}^{-1}\right)  T=J_{\beta,T}+\sum
_{\gamma\vartriangleleft\beta,T^{\prime}\in\mathcal{Y}\left(  \tau\right)
}u\left(  \beta,\gamma,T,T^{\prime};\kappa\right)  J_{\gamma,T^{\prime}},
\]
where $u\left(  \beta,\gamma,T,T^{\prime};\kappa\right)  $ is a rational
function of $\kappa$. By hypothesis for each $\gamma\vartriangleleft\beta$ and
$T^{\prime}\in\mathcal{Y}\left(  \tau\right)  $ there is an index $i\left[
\gamma,T^{\prime}\right]  $ such that $\zeta_{\beta,T}^{\prime}\left(
i\left[  \gamma,T^{\prime}\right]  \right)  \neq\zeta_{\gamma,T^{\prime}%
}^{\prime}\left(  i\left[  \gamma,T^{\prime}\right]  \right)  $ at
$\kappa=\kappa_{0}$. Recall that the generic spectral vector $\zeta
_{\gamma,T^{\prime}}^{\prime}$ uniquely determines $\left(  \gamma,T^{\prime
}\right)  $, since $\left[  \gamma_{i}\right]  _{i=1}^{N}$ is found from the
coefficients of $\frac{1}{\kappa}$ and the remaining terms of $\zeta
_{\gamma,T^{\prime}}^{\prime}$ determine the content vector of $T^{\prime}$.
Define an operator on $\mathcal{P}_{\tau}$ by%
\[
\mathcal{T}_{\beta,T}=\prod\limits_{\gamma\vartriangleleft\beta,T^{\prime}%
\in\mathcal{Y}\left(  \tau\right)  }\frac{\mathcal{U}_{i\left[  \gamma
,T^{\prime}\right]  }^{\prime}-\zeta_{\gamma,T^{\prime}}^{\prime}\left(
i\left[  \gamma,T^{\prime}\right]  \right)  }{\zeta_{\beta,T}^{\prime}\left(
i\left[  \gamma,T^{\prime}\right]  \right)  -\zeta_{\gamma,T^{\prime}}%
^{\prime}\left(  i\left[  \gamma,T^{\prime}\right]  \right)  }.
\]
Then $\mathcal{T}_{\beta,T}$ annihilates each $J_{\gamma,T^{\prime}}$ with
$\gamma\vartriangleleft\beta$ and maps $J_{\beta,T}$ to itself. Thus
$\mathcal{T}_{\beta,T}\left(  x^{\beta}\otimes\tau\left(  r_{\beta}\right)
T\right)  =J_{\beta,T}$ and by construction the right hand side has no poles
at $\kappa=\kappa_{0}$.
\end{proof}

The condition in the Proposition is sufficient, not necessary. There is an
example in the concluding remarks to support this statement.

We introduce a simple tool for the analysis of a pair $\beta,T$, namely the
tableau $\mathcal{X}_{\beta,T}$ with the entries being pairs $\left(
i,\beta_{r_{\beta}\left(  i\right)  ^{-1}}\right)  =\left(  i,\beta_{i}%
^{+}\right)  $ such that the tableau of just the first entries coincides with
$T$, that is, if $T\left[  a,b\right]  =i$ then $\mathcal{X}_{\beta,T}\left[
a,b\right]  =\left(  i,\beta_{i}^{+}\right)  $. As example let
\begin{align*}
T  &  =%
\begin{bmatrix}
12 & 11 & 10 & 6\\
9 & 8 & 5 & 2\\
7 & 4 & 3 & 1
\end{bmatrix}
,\beta=\left(  120201303121\right)  ,\beta^{+}=\left(  332221111000\right) \\
\mathcal{X}_{\beta,T}  &  =%
\begin{bmatrix}
\left(  12,0\right)  & \left(  11,0\right)  & \left(  10,0\right)  & \left(
6,1\right) \\
\left(  9,1\right)  & \left(  8,1\right)  & \left(  5,2\right)  & \left(
2,3\right) \\
\left(  7,1\right)  & \left(  4,2\right)  & \left(  3,2\right)  & \left(
1,3\right)
\end{bmatrix}
.
\end{align*}
The tableau $\mathcal{X}_{\beta,T}$ has order properties: in each row the
first entries decrease and the second entries nondecrease (weakly increase),
and the same holds for each column.

The first part is to assume $\lambda\trianglerighteq\beta$ and $\left(
m+1\right)  \beta_{i}+c\left(  r_{\beta}\left(  i\right)  ,T\right)  =c\left(
i,S_{0}\right)  $ (called the \textit{fundamental equation}) for $1\leq
i\leq2mk$ and to deduce that $\beta=\lambda$ and $T=T_{0}$

Our approach to the uniqueness proofs is to work one brick at a time, and in
each brick alternating between even and odd indices showing the values of
$\beta_{i}$ and $T^{\prime}$ agree with those of $\lambda,T_{0}$.. For each
cell we use the fundamental equation and the order properties of
$\mathcal{X}_{\beta,T}$ to set up inequalities which lead to a contradiction
if $\beta_{i}\neq\lambda_{i}$.

\begin{theorem}
\label{uniqlb!}Suppose $\left(  \beta,T\right)  \in\mathbb{N}_{0}^{2mk}%
\times\mathcal{Y}\left(  \tau\right)  $ , $\beta\trianglelefteq\lambda$ and
$\left(  m+2\right)  \beta_{j}+c\left(  r_{\beta}\left(  j\right)  ,T\right)
=c\left(  j,S_{0}\right)  $ for $1\leq j\leq2mk$ then $\beta=\lambda$ and
$T=T_{0}$.
\end{theorem}

\begin{proof}
This is an inductive argument alternating between even and odd indices to
prove the desired equalities for brick $B_{0}$. Then the argument is applied
to the tableaux with one less brick. Suppose we have shown $\beta_{2m-j}=0$
for $0\leq j\leq2n-1\leq2m-3$ (thus $n\leq m-1$) and $T\left[  1,i+1\right]
=2mk-2i$ for $0\leq i\leq n$ and $T\left[  2,i+1\right]  =2mk-2i-1$ for $0\leq
i\leq n-1$. The start of the induction is $n=0$ so the previous conditions are
vacuous. Suppose $\beta_{2mk-2n}=\ell$, $r_{\beta}\left(  2mk-2n\right)
=\rho$ and $\mathcal{X}_{\beta,T}\left[  a,b\right]  =\left(  \rho
,\ell\right)  $. Then%
\begin{align*}
\ell\left(  m+2\right)  +c\left(  \rho,T\right)   &  =c\left(  2mk-2n,S_{0}%
\right)  =n,\\
b-a  &  =c\left(  \rho,T\right)  =n-\ell\left(  m+2\right)  ,\\
a  &  =b-n+\ell\left(  m+2\right) \\
&  \geq\left(  \ell-1\right)  \left(  m+2\right)  +m+3-n.
\end{align*}
Thus if $\ell>0$ and $n\leq m-1$ then $a\geq4$. Let $\mathcal{X}_{\beta
,T}\left[  a-1,b\right]  =\left(  d,j\right)  $ with $d>\rho,$ $j<\ell$ and
$r_{\beta}\left(  e\right)  =d$ (thus $\beta_{e}=j$); furthermore $e<2mk-2n$
since $a-1\geq3$. Then%
\begin{align*}
j\left(  m+2\right)  +c\left(  d,T\right)   &  =j\left(  m+2\right)
+b-a+1=c\left(  e,S_{0}\right) \\
c\left(  e,S_{0}\right)   &  =j\left(  m+2\right)  +n+1-\ell\left(  m+2\right)
\\
&  =n+1-\left(  \ell-j\right)  \left(  m+2\right)  \leq n-m-1
\end{align*}
But $\min\left(  c\left(  i,S_{0}\right)  :i<2mk-2n\right)  =n-1$ and there is
a contradiction. Thus $\beta_{2mk-2n}=0$, $\rho=2mk-2n$ and $T\left[
1,n+1\right]  =2mk-2n$ (because the entry $2mk-2n$ has to be adjacent to
$2mk-2n+2$). The start $n=0$ forces $T\left[  1,1\right]  =2mk$. The last step
is with $n=m-1$ and results in $T\left[  1,m\right]  =2mk-2m+2$.

Suppose we have shown $\beta_{2mk-j}=0$ for $0\leq j\leq2n\leq2m-2$ and
$T\left[  1,i+1\right]  =2mk-2i$ for $0\leq i\leq n$ and $T\left[
2,i+1\right]  =2mk-2i-1$ for $0\leq i<n$ (the first step is with $n=0$).
Suppose $\beta_{2mk-2n-1}=\ell$, $r_{\beta}\left(  2mk-2n-1\right)  =\rho$ and
$\mathcal{X}_{\beta,T}\left[  a,b\right]  =\left(  \rho,\ell\right)  $ then%
\begin{align*}
\ell\left(  m+2\right)  +c\left(  \rho,T\right)   &  =c\left(  2mk-2n-1,S_{0}%
\right)  =n-1,\\
b-a  &  =c\left(  \rho,T\right)  =n-1-\ell\left(  m+2\right)  ,\\
a  &  =b-n+1+\ell\left(  m+2\right) \\
&  \geq\left(  \ell-1\right)  \left(  m+2\right)  +m-n+4.
\end{align*}
If $\ell>0$ and $n\leq m-1$ then $a\geq4$ and $\mathcal{X}_{\beta,T}\left[
a-1,b\right]  =\left(  d,j\right)  $ with $d>\rho$, $j<\ell$ (because
$\beta_{2mk-2n-1}$ is the last appearance of $\ell$ in $\beta$) and $r_{\beta
}\left(  e\right)  =d$ (thus $\beta_{e}=j$). From $a-1\geq3$ it follows that
$e<2mk-2n-1$. Then
\begin{align*}
\left(  m+2\right)  j+c\left(  d,T\right)   &  =\left(  m+2\right)
j+b-a+1=c\left(  e,S_{0}\right) \\
c\left(  e,S_{0}\right)   &  =\left(  m+2\right)  j+n-\ell\left(  m+2\right)
\\
&  =n-\left(  \ell-j\right)  \left(  m+2\right)  \leq n-m-2.
\end{align*}
But $\min\left(  c\left(  i,S_{0}\right)  :i<2mk-2n-1\right)  =n>n-m-2$ and
this is a contradiction Thus $\beta_{2mk-2n-1}=0$ and $T\left[  2,n+1\right]
=2mk-2n-1$ (because $\left[  2,n+1\right]  $ is the only remaining cell in the
first two rows with content $n-1$). The last step of the induction is for
$n=m-1$.

Replace the original problem by a smaller one: let $\lambda^{\prime}=\left[
\lambda_{i}-1\right]  _{i=1}^{2m\left(  k-1\right)  },\beta^{\prime}=\left[
\beta_{i}-1\right]  _{i=1}^{2m\left(  k-1\right)  }$ the tableau
$S_{0}^{\prime}$ of shape $2\times m\left(  k-1\right)  $ with entries
$S_{0}^{^{\prime}}\left[  i,j\right]  =S_{0}\left[  i,j+m\right]  $ for
$i=1,2$ and $1\leq j\leq m\left(  k-1\right)  $ and the tableaux
$T_{0}^{\prime}$ and $T^{\prime}$ of shape $2\left(  k-1\right)  \times m$
with entries $T_{0}^{\prime}\left[  i,j\right]  =T_{0}\left[  i+2,j\right]
,T^{\prime}\left[  i,j\right]  =T\left[  i+2,j\right]  $ for $1\leq
i\leq2\left(  k-1\right)  $ and $1\leq j\leq m.$ The consequences of these
definitions are with $1\leq i\leq2m\left(  k-1\right)  $
\begin{align*}
c\left(  i,T^{\prime}\right)   &  =c\left(  i,T\right)  +2,c\left(
i,T_{0}^{\prime}\right)  =c\left(  i,T_{0}\right)  +2,\\
c\left(  i,S_{o}^{\prime}\right)   &  =c\left(  i,S_{0}\right)  -m,\\
r_{\beta^{\prime}}\left(  i\right)   &  =r_{\beta}\left(  i\right)
,\lambda^{\prime}\trianglerighteq\beta^{\prime}.
\end{align*}
Then%
\begin{align*}
\left(  m+2\right)  \beta_{i}^{\prime}+c\left(  r_{\beta^{\prime}}\left(
i\right)  ,T^{\prime}\right)   &  =\left(  m+2\right)  \left(  \beta
_{i}-1\right)  +c\left(  r_{\beta}\left(  i\right)  ,T\right)  +2\\
&  =\left(  m+2\right)  \beta_{i}+c\left(  r_{\beta}\left(  i\right)
,T\right)  -m\\
&  =c\left(  i,S_{0}\right)  -m=c\left(  i,S_{0}^{\prime}\right)  .
\end{align*}
and the same argument as before shows that $\left(  \beta,T\right)  $ agrees
with $\left(  \lambda,T_{0}\right)  $ in the first four rows (the first two
bricks). Repeat this process $\left(  k-1\right)  $ times arriving at
$\beta_{i}^{^{\prime\prime}}=0$ for $1\leq i\leq2m$ and the entries of the
remaining $T^{\prime\prime}$ are $2m,2m-1,\ldots,1$ entered column by column%
\[%
\begin{bmatrix}
2m & \cdots & 4 & 2\\
2m-1 & \cdots & 3 & 1
\end{bmatrix}
.
\]
Thus $\left(  \beta,T\right)  =\left(  \lambda,T_{0}\right)  $ and the
spectral vector of $\left(  \lambda,T\right)  $ is unique.
\end{proof}

We set up the same argument for removing the last brick $B_{k-1}$. Intuitively
this is already done: rotate the tableaux through $180^{\cdot}$ and replace
the entry $r$ by $2mk+1-r$. This idea guides the proof. The property
$\beta\trianglelefteq\lambda$ implies $\beta_{i}\leq k-1$ for all $i$. Here
the inductive argument alternates between odd and even indices.

\begin{proof}
\label{uniqlb2}(Theorem \ref{uniqlb!} alternate): Suppose we have shown
$\beta_{j}=k-1$ for $1\leq j\leq2n$, $T\left[  2k,m-j\right]  =2j+1$ for
$0\leq j\leq n-1$ and $T\left[  2k-1,m-j\right]  =2j+2$ for $1\leq j\leq n-1$
(the first step is at $n=0$ with vacuous conditions on $T$, the last at
$n=m-1$). Suppose $\beta_{2n+1}=\ell$, $r_{\beta}\left(  2n+1\right)  =\rho$
and $\mathcal{X}_{\beta,T}\left[  a,b\right]  =\left(  \rho,\ell\right)  $
then (using $b\leq m$)%
\begin{align*}
\ell\left(  m+2\right)  +c\left(  \rho,T\right)   &  =c\left(  2n+1,S_{0}%
\right)  =mk-n-2,\\
b-a  &  =mk-n-2-\ell\left(  m+2\right)  ,\\
a  &  =b-mk+n+2+\ell\left(  m+2\right) \\
&  \leq n+2k-\left(  k-1-\ell\right)  \left(  m+2\right)  .
\end{align*}
If $\ell<k-1$ and $n\leq m-1$ then $a\leq2k-3$. Let $\mathcal{X}_{\beta
,T}\left[  a+1,b\right]  =\left(  d,j\right)  $ with $d<\rho$, $j>\ell$, and
$r_{\beta}\left(  e\right)  =d$ (so that $\beta_{e}=j$ ; furthermore $e>2n+1$
since $a+1\leq2k-2$, then%
\begin{align*}
\left(  m+2\right)  j+c\left(  d,T\right)   &  =\left(  m+2\right)
j+b-a-1=c\left(  e,S_{0}\right) \\
c\left(  e,S_{0}\right)   &  =mk-n-3-\ell\left(  m+2\right)  +j\left(
m+2\right) \\
&  \geq mk+m-n-1
\end{align*}
But $\max\left\{  c\left(  i,S_{0}\right)  :i>2n+1\right\}  =mk-n-1<mk+m-n-1$,
a contradiction, thus $\beta_{2n+1}=k-1$ and $T\left[  2k,m-n\right]  =2n+1$
(because this is the only possible cell in the last two rows with entry
$2n+1$). The start is $T\left[  2k,m\right]  =1$ (forced by definition of
RSYT) and $\beta_{1}=k-1$. The last step results in $T\left[  2k,1\right]
=2m-1$.

Suppose we have shown $\beta_{j}=k-1$ for $1\leq j\leq2n-1$, $T\left[
2k,m-j\right]  =2j+1$ for $0\leq j\leq n-1$ and $T\left[  2k-1,m-j\right]
=2j+2$ for $0\leq j\leq n-2$ (the first step is at $n=1$, the last at $n=m$).
Suppose $\beta_{2n}=\ell$ , $r_{\beta}\left(  2n\right)  =\rho$ and
$\mathcal{X}_{\beta,T}\left[  a,b\right]  =\left(  \rho,\ell\right)  $ then%
\begin{align*}
\ell\left(  m+2\right)  +c\left(  \rho,T\right)   &  =c\left(  2n,S_{0}%
\right)  =mk-n,\\
b-a  &  =mk-n-\ell\left(  m+2\right)  ,\\
a  &  =b-mk+n+\ell\left(  m+2\right) \\
&  \leq n+2k+2-\left(  k-1-\ell\right)  \left(  m+2\right)  .
\end{align*}
Thus if $\ell<k-1$ and $n\leq m$ then $a\leq2k-3$ . Let $\mathcal{X}_{\beta
,T}\left[  a+1,b\right]  =\left(  d,j\right)  $ with $d<\rho$, $j>\ell$, and
$r_{\beta}\left(  e\right)  =d$ (so that $\beta_{e}=j$ ) then%
\begin{align*}
\left(  m+2\right)  j+c\left(  d,T\right)   &  =b-a-1+\left(  m+2\right)
j=c\left(  e,S_{0}\right) \\
c\left(  e,S_{0}\right)   &  =mk-n-1-\ell\left(  m+2\right)  +j\left(
m+2\right) \\
&  \geq mk+m-n+1
\end{align*}
But $\max\left\{  c\left(  i,S_{0}\right)  :i>2n\right\}  =mk-n-1<mk+m-n+1$
and there is a contradiction. Thus $\beta_{2n}=k-1,r_{b}\left(  2n\right)
=2n$ and $T\left[  2k-1,m+1-n\right]  =2n$, because this is the only cell in
the last two rows with content $m-n-2k+2$.

The inductive process concludes by showing $\beta_{i}=k-1$ for $1\leq i\leq2m$
and $T\left[  i,j\right]  =T_{0}\left[  i,j\right]  $ for $i=2k-1,2k$ and
$1\leq j\leq m$. As before the original problem can be reduced to a smaller
one by removing the last brick. This is implemented by defining $\lambda
^{\prime},\beta^{\prime},S_{0}^{\prime},T_{0}^{\prime},T^{\prime}$ as follows:%
\begin{align*}
\lambda_{i}^{\prime}  &  =\lambda_{i+2m},\beta_{i}^{\prime}=\beta
_{i+2m},r_{\beta^{\prime}}\left(  i\right)  =r_{\beta}\left(  i+2m\right)
-2m,1\leq i\leq2m\left(  k-1\right)  ,\\
S_{0}^{\prime}\left[  i,j\right]   &  =S_{0}\left[  i,j\right]
-2m,i=1,2,1\leq j\leq m\left(  k-1\right)  ,\\
T_{0}^{\prime}\left[  i,j\right]   &  =T_{0}\left[  i,j\right]  -2m,T^{\prime
}\left[  i,j\right]  =T\left[  i,j\right]  -2m,1\leq i\leq2\left(  k-1\right)
,1\leq j\leq m.
\end{align*}
Clearly $\lambda^{\prime}\trianglerighteq\beta^{\prime}$ and the hypothesis
$\left(  m+2\right)  \beta_{j}^{^{\prime}}+c\left(  r_{\beta^{\prime}}\left(
j\right)  ,T^{\prime}\right)  =c\left(  j,S_{0}^{\prime}\right)  $ holds for
$1\leq j\leq2m\left(  k-1\right)  $. So the bricks can be removed in the order
$k-1,k-2,\ldots$. At the end there is only one brick $B_{0},$all $\beta
_{i}=0,r_{\beta}\left(  i\right)  =i$ and $c\left(  i,T\right)  =c\left(
i,S_{0}\right)  $ (for $2mk-2m<i\leq2mk$ and the corresponding parts (brick
$B_{0}$) of $T$ and $S_{0}$ are identical, and thus to $T_{0}$.
\end{proof}

The second part is to prove uniqueness for the spectral vectors derived from
the content vectors of the tableaux $S_{\left(  j,n\right)  }$ (see Definition
\ref{defSjn}) where $n=ms$, at the edge of brick $B_{s-1}$ adjacent to the
edge of $B_{s}$ .To prove this we use the previous arguments to remove the
bricks above and below bricks $s-1$ and $s$ leaving us with a straightforward
argument where only two values of $\lambda$ play a part. To obtain the
hypothetically unique spectral vectors we apply reflections to $J_{\lambda
,T_{0}}$. For brevity let $i_{0}=2m\left(  k-s\right)  $. First compute
$s_{i_{0}}J_{\lambda,T_{0}}$, a nonzero multiple of $J_{s_{0}\lambda,T_{0}}$;
this is a permissible step and hence this polynomial is defined for
$\kappa=\frac{1}{m+2}$. Then form $s_{i_{0}+1}s_{i_{0}}J_{\lambda,T_{0}}$
which produces the polynomial labeled by $\left(  \alpha^{\left(  1\right)
},T_{0}\right)  $ whose spectral vector equals the content vector of
$S_{\left(  1,ms\right)  }$. Also form $s_{i_{0}-1}s_{i_{0}}J_{\lambda,T_{0}}%
$, with label $\left(  \alpha^{\left(  2\right)  },T_{0}\right)  $ associated
with $S_{\left(  2,ms\right)  }$. Here are tables of values of $\lambda
,s_{i_{0}}\lambda$, $\alpha^{\left(  1\right)  }=s_{i_{0}+1}s_{i_{0}}\lambda$,
$\alpha^{\left(  2\right)  }=s_{i_{0}-1}s_{i_{0}}\lambda$ in the zone of
relevance $\left(  i_{0}-1\leq i\leq i_{0}+2\right)  $;
\[%
\begin{vmatrix}
i= & i_{0}-1 & i_{0} & i_{0}+1 & i_{0}+2\\
\lambda & s & s & s-1 & s-1\\
s_{i_{0}}\lambda & s & s-1 & s & s-1\\
\alpha^{\left(  1\right)  } & s & s-1 & s-1 & s\\
\alpha^{\left(  2\right)  } & s-1 & s & s & s-1
\end{vmatrix}
\]
and the corresponding spectral vectors $\beta_{i}\left(  m+2\right)  +c\left(
r_{\beta}\left(  i\right)  ,T_{0}\right)  $, denoted by $v\left(
\beta,i\right)  $ for convenience,%
\[%
\begin{vmatrix}
i= & i_{0}-1 & i_{0} & i_{0}+1 & i_{0}+2\\
v\left(  \lambda,\cdot\right)  & sm-1 & sm & sm-2 & sm-1\\
v\left(  s_{i_{0}}\lambda,\cdot\right)  & sm-1 & sm-2 & sm & sm-1\\
v\left(  \alpha^{\left(  1\right)  },\cdot\right)  & sm-1 & sm-2 & sm-1 & sm\\
v\left(  \alpha^{\left(  2\right)  },\cdot\right)  & sm-2 & sm-1 & sm & sm-1
\end{vmatrix}
.
\]
The respective cells in $T_{0}$ are $T_{0}\left[  2s+2,1\right]
=i_{0}-1,T_{0}\left[  2s+1,1\right]  =i_{0},T_{0}\left[  2s,m\right]
=i_{0}+1,T_{0}\left[  2s-1,m\right]  =i_{0}+2$ with respective contents
$-1-2s,-2s,m-2s,m+1-2s$. Except for these four locations $v\left(
\lambda,i\right)  =v\left(  \alpha^{\left(  1\right)  },i\right)  =v\left(
a^{\left(  2\right)  },i\right)  =c\left(  i,S_{0}\right)  $ so in the bricks
$B_{j}$ for $0\leq j<s-1$ and $s<j\leq k-1$ the previous proofs can be applied
; the various $\max\left\{  v\left(  \lambda,i\right)  :i>b\right\}  $ and
$\min\left\{  v\left(  \lambda,i\right)  :i<b\right\}  $ values apply verbatim.

\begin{theorem}
Suppose $u=1$ or $2$ and $\left(  \beta,T\right)  \in\mathbb{N}_{0}%
^{2mk}\times\mathcal{Y}\left(  \tau\right)  $ such that $\beta\trianglelefteq
\alpha^{\left(  u\right)  }$ and $\left(  m+2\right)  \beta_{j}+c\left(
r_{\beta}\left(  i\right)  ,T\right)  =v\left(  \alpha^{\left(  u\right)
},i\right)  $ for $1\leq i\leq2mk$ then $\left(  \beta,T\right)  =\left(
\alpha^{\left(  u\right)  },T_{0}\right)  $.
\end{theorem}

\begin{proof}
By the previous arguments we show $T\left[  i,j\right]  =T_{0}\left[
i,j\right]  $ for $1\leq j\leq m$, $1\leq i\leq2s-2$ (using the proof for
Theorem \ref{uniqlb!}) and $2s+3\leq i\leq2k$ (using the alternate proof
\ref{uniqlb2}).. This leaves just two bricks and we can assume $s=1,k=2$.
Reducing $v\left(  \alpha^{\left(  u\right)  },\cdot\right)  $ to $s=1,k=2$
results in (with $1\leq i\leq2m$ and $i\neq m,m+1$)%
\begin{align*}
v\left(  \alpha^{\left(  u\right)  },2i\right)   &  =2m-i,\\
v\left(  \alpha^{\left(  u\right)  },2i-1\right)   &  =2m-i-1,\\
\left[  v\left(  \alpha^{\left(  1\right)  },i\right)  \right]  _{i=2m-1}%
^{2m+2}  &  =\left[  m-2,m-1,m,m-1\right]  ,\\
\left[  v\left(  \alpha^{\left(  2\right)  },i\right)  \right]  _{i=2m-1}%
^{2m+2}  &  =\left[  m-1,m-2,m-1,m\right]  .
\end{align*}
The property $\beta\trianglelefteq\alpha^{\left(  u\right)  }$ implies that
$\beta$ is a permutation of $\left(  1^{2m},0^{2m}\right)  $. The entries
$2m+1,2m+2,\ldots,4m$ in $T$ are all in $B_{0}$ and the entries $1,2,\ldots
,2m$ are in $B_{1}$. This follows from $\mathcal{X}_{\beta,T}\left[
a,b\right]  =\left(  r,0\right)  $ implies $1\leq a\leq2$, for if $a=3$ or $4$
then the ordering property of $\mathcal{X}_{\beta,T}$ implies $\mathcal{X}%
_{\beta,T}\left[  3,1\right]  =\left(  r^{\prime},0\right)  $ or
$\mathcal{X}_{\beta,T}\left[  4,1\right]  =\left(  r^{\prime},0\right)  $ but
$c\left(  r^{\prime},T\right)  =-2$ or $-3$ is impossible (as values of
$v\left(  \alpha^{\left(  u\right)  },i\right)  $). Thus the $2m$ pairs
$\left\{  \left(  r,0\right)  :2m+1\leq r\leq4m\right\}  $ fill $\left\{
\mathcal{X}_{\beta,T}\left[  a,b\right]  :a=1,2,1\leq b\leq m\right\}  $. The
next few steps are for $m\geq2$; if $m=1$ then there are just $4$ cells left
and the last part of the proof suffices. By using the previous arguments we
show $\beta_{i}=1$ for $1\leq i\leq2m-2$ and $\beta_{i}=0$ for $2m+3\leq
i\leq4m$, and also that for $1\leq j\leq m-1~$and $m+2\leq j\leq2m$%
\begin{align*}
T\left[  1,j\right]   &  =T_{0}\left[  1,j\right]  =4m+2-2j,\\
T\left[  2,j\right]   &  =T_{0}\left[  2,j\right]  =4m+1-2j,
\end{align*}
and for $2\leq j\leq m$
\begin{align*}
T\left[  3,j\right]   &  =T_{0}\left[  3,j\right]  =2m+2-2j\\
T\left[  4,j\right]   &  =T_{0}\left[  4,j\right]  =2m+1-2j.
\end{align*}
As example of the steps of the proof let $\beta_{1}=\ell$ then $\left(
m+2\right)  \ell+c\left(  r_{\beta}\left(  1\right)  ,T\right)  =2m-2$ and let
$\mathcal{X}_{\beta,T}\left[  a,b\right]  =\left(  r_{\beta}\left(  1\right)
,\ell\right)  $. Then
\begin{align*}
b-a  &  =2m-2-\left(  m+2\right)  \ell,\\
a  &  =b-2m+2+\left(  m+2\right)  \ell\\
&  \leq2-m+\left(  m+2\right)  \ell
\end{align*}
But if $\ell=0$ then $a\leq0$ (since $m\geq2$) which is impossible; thus
$\beta_{1}=1,a=4$ and $T\left[  4,m\right]  =1.$ Similarly consider
$\beta_{4m}=\ell$ and $\mathcal{X}_{\beta,T}\left[  a,b\right]  =\left(
r_{\beta}\left(  4m\right)  ,\ell\right)  $ then%
\begin{align*}
0  &  =\left(  m+2\right)  \ell+c\left(  r_{\beta}\left(  4m\right)
,T\right)  ,\\
a  &  =b+\left(  m+2\right)  \ell\geq1+\left(  m+2\right)  \ell,
\end{align*}
but if $\ell=1$ then $a\geq5$ which is impossible; thus $\beta_{4m}=0$ and
$T\left[  1,1\right]  =4m$.

All but four entries have been accounted for and thus $T\left[  2,m\right]
=2m+1$, the rank of the first zero in $\beta$, and $T\left[  3,1\right]
=2m,$the rank of the last $1$ in $\beta$. Thus $T=T_{0}$. The relevant part of
the content vector is
\[
\left[  c\left(  i,T\right)  \right]  _{i=2m-1}^{2m+2}=\left[
-3,-2,m-2,m-1\right]
\]
The remaining equations are%
\[%
\begin{array}
[c]{ccc}
& \alpha^{\left(  1\right)  } & \alpha^{\left(  2\right)  }\\
\left(  m+2\right)  \beta_{2m-1}+c\left(  r_{\beta}\left(  2m-1\right)
,T\right)  = & m-1 & m-2\\
\left(  m+2\right)  \beta_{2m}+c\left(  r_{\beta}\left(  2m\right)  ,T\right)
= & m-2 & m-1\\
\left(  m+2\right)  \beta_{2m+1}+c\left(  r_{\beta}\left(  2m+1\right)
,T\right)  = & m-1 & m\\
\left(  m+2\right)  \beta_{2m+2}+c\left(  r_{\beta}\left(  2m+2\right)
,T\right)  = & m & m-1
\end{array}
.
\]
Let $\beta_{j_{1}}=\beta_{j_{2}}=1$ with $2m-1\leq j_{1}<j_{2}\leq2m+2$. Then
$r_{\beta}\left(  j_{1}\right)  =2m-1,m+2+c\left(  r_{\beta}\left(
j_{1}\right)  \right)  =m-1$ and $r_{\beta}\left(  j_{2}\right)
=2m,m+2+c\left(  r_{\beta}\left(  j_{2}\right)  \right)  =m$. Let
$\beta_{j_{3}}=\beta_{j_{4}}=0$ with $2m-1\leq j_{3}<j_{4}\leq2m+2$. Then
$r_{\beta}\left(  j_{3}\right)  =2m+1,c\left(  r_{\beta}\left(  j_{3}\right)
\right)  =m-2$ and $r_{\beta}\left(  j_{4}\right)  =2m+2,c\left(  r_{\beta
}\left(  j_{4}\right)  \right)  =m-1$.

\textit{Case} $\alpha^{\left(  1\right)  }$: From the table we see that
$j_{3}=2m$ and $j_{2}=2m+2$. This implies $j_{4}=2m+1$ and.$j_{1}=2m-1$ Thus
$\beta=\alpha^{\left(  1\right)  }$, with central entries $\left(
1,0,0,1\right)  .$

\textit{Case} $\alpha^{\left(  2\right)  }:$ From the table we see that
$j_{3}=2m-1$ and $j_{2}=2m+1$. This implies $j_{1}=2m$ and $j_{4}=2m+2$. Thus
$\beta=\alpha^{\left(  2\right)  }$, with $\left(  0,1,1,0\right)  $ being the
central entries.

This concludes the proof.
\end{proof}

\section{Maps of standard modules\label{Maps}}

The algebra generated by $\mathcal{D}_{i}$ and multiplication by $x_{i}$ for
$1\leq i\leq2mk$ along with $w\in\mathcal{S}_{2mk}$ is the rational Cherednik
algebra (type $A_{2mk-1}$) and $\mathcal{P}_{\tau}$ is called the standard
module associated with $\tau$, denoted $\Delta_{\kappa}\left(  \tau\right)  $.
In this section we construct a homomorphism from the module $\mathcal{P}%
_{\sigma}$ to $\mathcal{P}_{\tau}$ when the parameter $\kappa=\frac{1}{m+2}$.
In the notation of Definition \ref{S2BT} for each $S\in\mathcal{Y}\left(
\sigma\right)  $ there is a pair $\left(  \beta\left\{  S\right\}  ,T\left\{
S\right\}  \right)  $ such that the spectral vector $\zeta_{\beta\left\{
S\right\}  ,T\left\{  S\right\}  }^{\prime}\left(  i\right)  =c\left(
i,S\right)  $ for $1\leq i\leq2mk$ at $\kappa=\frac{1}{m+2}$. However the
polynomials $J_{\beta\left\{  S\right\}  ,T\left\{  S\right\}  }$ need to be
rescaled so that they transform under $w$ with the same matrix as
$\mathcal{Y}\left(  \sigma\right)  $. Recall the formula for $\left\Vert
S\right\Vert ^{2}$ which is derived from the requirement that $\left\{
S:S\in\mathcal{Y}\left(  \sigma\right)  \right\}  $ is an orthogonal basis and
each $\sigma\left(  w\right)  $ is an isometry (and we use this requirement
for $\left\Vert J_{\beta\left\{  S\right\}  ,T\left\{  S\right\}  }\right\Vert
^{2}$ as well)%
\[
\left\Vert S\right\Vert ^{2}=\prod\limits_{\substack{1\leq i<j\leq
2km\\c\left(  i,S\right)  -c\left(  j,S\right)  \leq-2}}\left(  1-\frac
{1}{\left(  c\left(  i,S\right)  -c\left(  j,S\right)  \right)  ^{2}}\right)
.
\]
By the construction of $S_{0}$ (column by column) either $i$ is odd and $j>i$
implies $\operatorname{col}\left(  j,S_{0}\right)  \leq\operatorname{col}%
\left(  i,S_{0}\right)  $ and $c\left(  i,S_{0}\right)  -c\left(
j,S_{0}\right)  \geq-1$ or $i$ is even and $j>i$ implies $\operatorname{col}%
\left(  j,S_{0}\right)  <\operatorname{col}\left(  i,S_{0}\right)  $ and
$c\left(  i,S_{0}\right)  -c\left(  j,S_{0}\right)  \geq0$, thus $\left\Vert
S_{0}\right\Vert ^{2}=1$.

Suppose $\operatorname{row}\left(  i,S\right)  <\operatorname{row}\left(
i+1,S\right)  ,\operatorname{col}\left(  i,S\right)  >\operatorname{col}%
\left(  i+1,S\right)  $ so that%
\begin{align*}
c\left(  i,S\right)   &  =\operatorname{col}\left(  i,S\right)
-\operatorname{row}\left(  i,S\right)  \geq\left(  \operatorname{col}\left(
i+1,S\right)  +1\right)  -\left(  \operatorname{row}\left(  i+1,S\right)
-1\right) \\
&  =c\left(  i+1,S\right)  +2,
\end{align*}
and the transformation rule (\ref{ieqi1}) yields (with $b_{i}\left(  S\right)
=\left(  c\left(  i,S\right)  -c\left(  i+1,S\right)  \right)  ^{-1}$):%
\begin{align*}
\sigma\left(  s_{i}\right)  S  &  =S^{\left(  i\right)  }+b_{S}\left(
i\right)  S\\
\left\Vert S\right\Vert ^{2}  &  =\left\Vert \sigma\left(  s_{i}\right)
S\right\Vert ^{2}=\left\Vert S^{\left(  i\right)  }\right\Vert ^{2}%
+b_{S}\left(  i\right)  ^{2}\left\Vert S\right\Vert ^{2}\\
\left\Vert S^{\left(  i\right)  }\right\Vert ^{2}  &  =\left(  1-b_{S}\left(
i\right)  ^{2}\right)  \left\Vert S\right\Vert ^{2}.
\end{align*}
We need two rules for the NSJP $J_{\alpha,T}$:

\begin{enumerate}
\item $\alpha_{i}=\alpha_{i+1},j=r_{\alpha}\left(  i\right)  $ and $c\left(
j,T\right)  -c\left(  j+1,T\right)  \geq2$ (see (\ref{ieqi1}))%
\begin{align}
s_{i}J_{\alpha,T}  &  =b_{\alpha,T}\left(  i\right)  J_{\alpha,T}%
+J_{\alpha,T^{\left(  j\right)  }}\nonumber\\
\left\Vert J_{\alpha,T}\right\Vert ^{2}  &  =\left\Vert s_{i}J_{\alpha
,T}\right\Vert ^{2}=b_{\alpha,T}\left(  i\right)  ^{2}\left\Vert J_{\alpha
,T}\right\Vert ^{2}+\left\Vert J_{\alpha,T^{\left(  j\right)  }}\right\Vert
^{2}\nonumber\\
\left\Vert J_{\alpha,T^{\left(  j\right)  }}\right\Vert ^{2}  &  =\left(
1-b_{\alpha,T}\left(  i\right)  ^{2}\right)  \left\Vert J_{\alpha
,T}\right\Vert ^{2} \label{aieq1i1}%
\end{align}

\item $\alpha_{i}>\alpha_{i+1}$ (see (\ref{igti1}))%
\begin{align}
s_{i}J_{s_{i}\alpha,T} &  =-b_{\alpha,T}\left(  i\right)  J_{s_{i}\alpha
,T}+J_{\alpha,T}\nonumber\\
\left\Vert J_{s_{i}\alpha,T}\right\Vert ^{2} &  =\left\Vert s_{i}%
J_{s_{i}\alpha,T}\right\Vert ^{2}=b_{\alpha,T}\left(  i\right)  ^{2}\left\Vert
J_{s_{i}\alpha,T}\right\Vert ^{2}+\left\Vert J_{\alpha,T}\right\Vert
^{2}\nonumber\\
\left\Vert J_{s_{i}\alpha,T}\right\Vert ^{2} &  =\left(  1-b_{\alpha,T}\left(
i\right)  ^{2}\right)  ^{-1}\left\Vert J_{\alpha,T}\right\Vert ^{2}%
.\label{aigtai1}%
\end{align}

\end{enumerate}

Recall the abbreviation $\pi\left\{  S\right\}  =\beta\left\{  S\right\}
,T\left\{  S\right\}  $. The following discussion of $\left\Vert J_{\alpha
,T}\right\Vert ^{2}$ applies only to $\mathrm{span}\left\{  J_{\pi\left\{
S\right\}  }:S\in\mathcal{Y}\left(  \sigma\right)  \right\}  $ at
$\kappa=\frac{1}{m+2}$, which is an irreducible $\mathcal{S}_{2mk}$-module,
isomorphic to $V_{\sigma}$. We use the normalization ($\lambda=\beta\left\{
S_{0}\right\}  ,T_{0}=T\left\{  S_{0}\right\}  $)
\[
\left\Vert J_{\lambda,T_{0}}\right\Vert ^{2}=1
\]
and this determines the other norms. .

\begin{definition}
For $S\in\mathcal{Y}\left(  \sigma\right)  $ let $\gamma_{S}=\left\Vert
S\right\Vert /\left\Vert J_{\pi\left\{  S\right\}  }\right\Vert $. By
convention $\gamma_{S_{0}}=1$.
\end{definition}

\begin{proposition}
Suppose $S\in\mathcal{Y}\left(  \sigma\right)  $ then%
\[
\gamma_{S}=\prod\limits_{\substack{1\leq i<j\leq2mk\\\beta\left\{  S\right\}
_{i}<\beta\left\{  S\right\}  _{j}}}\left(  1-\frac{1}{\left(  c\left(
i,S\right)  -c\left(  j,S\right)  \right)  ^{2}}\right)  .
\]

\end{proposition}

\begin{proof}
We argue by induction on $\mathrm{inv}\left(  S\right)  $ (see (\ref{invS})).
Suppose the formula holds for each $S\in\mathcal{Y}\left(  \sigma\right)  $
with $\mathrm{inv}\left(  S\right)  \geq\mathrm{inv}\left(  S_{0}\right)  -u$
for some $u$; the start is $u=0$. Suppose $\mathrm{inv}\left(  S^{\prime
}\right)  =\mathrm{inv}\left(  S_{0}\right)  -u-1$ and $c\left(  i,S^{\prime
}\right)  -c\left(  i+1,S^{\prime}\right)  \leq-2$, that is,
$\operatorname{row}\left(  i,S^{\prime}\right)  =2,\operatorname{row}\left(
i+1,S^{\prime}\right)  =1$ and $\operatorname{col}\left(  i,S^{\prime}\right)
<\operatorname{col}\left(  i+1,S^{\prime}\right)  $. Then $S^{\prime
}=S^{\left(  i\right)  }$ and $\mathrm{inv}\left(  S\right)  =\mathrm{inv}%
\left(  S_{0}\right)  -u$. Also $\left\Vert S^{\left(  i\right)  }\right\Vert
^{2}=\left(  1-b_{i}\left(  S\right)  ^{2}\right)  \left\Vert S\right\Vert
^{2}$. For convenience let $\beta=\beta\left\{  S\right\}  ,T=T\left\{
S\right\}  $. (Recall that $\zeta_{\beta\left\{  S\right\}  ,T\left\{
S\right\}  }^{\prime}\left(  j\right)  =c\left(  j,S\right)  $ for all $j$).
There are two cases for the relative locations of $i$ and $i+1$ in $S$.

If $i,i+1\in B_{\ell}$ for some $\ell$ then by definition $\beta_{i}%
=\beta_{i+1}=\ell$, $r_{\beta}\left(  i+1\right)  =j+1=r_{\beta}\left(
i\right)  +1$ and $j,j+1$ are in the same brick $B_{\ell}$ in $T$. Then
$T\left\{  S^{\left(  i\right)  }\right\}  =T^{\left(  j\right)  }%
,\beta\left\{  S^{\left(  i\right)  }\right\}  =\beta$. By formula
(\ref{aieq1i1})%
\begin{align*}
\left\Vert J_{\pi\left\{  S^{\left(  i\right)  }\right\}  ,}\right\Vert ^{2}
&  =\left(  1-b_{i}\left(  S\right)  ^{2}\right)  \left\Vert J_{\pi\left\{
S\right\}  }\right\Vert ^{2}\\
\gamma_{S^{\prime}}  &  =\gamma_{S}.
\end{align*}
Because $\beta_{i}=\beta_{i+1}$ the product in $\gamma_{S}$ is invariant under
the replacement $S\rightarrow S^{\left(  i\right)  }.$

If $i$ and $i+1$ are in different bricks then $\beta_{i}>\beta_{i+1}$ because
$\operatorname{col}\left(  i,S\right)  >\operatorname{col}\left(
i+1,S\right)  $. Then $\beta\left\{  S^{\prime}\right\}  =s_{i}\beta,T\left\{
S^{\prime}\right\}  =T$ and by formula (\ref{aigtai1})%
\begin{align*}
\left\Vert J_{s_{i}\beta,T}\right\Vert ^{2} &  =\left(  1-b_{\beta,T}\left(
i\right)  ^{2}\right)  ^{-1}\left\Vert J_{\beta,T}\right\Vert ^{2}\\
\frac{\left\Vert S^{\left(  i\right)  }\right\Vert ^{2}}{\left\Vert
J_{s_{i}\beta,T}\right\Vert ^{2}} &  =\left(  1-b_{\beta,T}\left(  i\right)
^{2}\right)  ^{2}\frac{\left\Vert S\right\Vert ^{2}}{\left\Vert J_{\beta
,T}\right\Vert ^{2}}\\
\gamma_{S^{\prime}} &  =\left(  1-b_{\beta,T}\left(  i\right)  ^{2}\right)
\gamma_{S}.
\end{align*}
The product in $\gamma_{S}$ is over pairs $\left(  a,b\right)  $ with $a<b$
and $\beta_{a}<\beta_{b}$. Changing $S$ to $S^{\left(  i\right)  }$ leaves the
pairs with $\left\{  a,b\right\}  \cap\left\{  i,i+1\right\}  =\varnothing$
alone and interchanges the pairs $\left(  a,i\right)  ,\left(  a,i+1\right)  $
and $\left(  i,b\right)  ,\left(  i+1,b\right)  $ respectively. The pair
$\left(  i,i+1\right)  $ is added to the product since $\beta\left\{
S^{\left(  i\right)  }\right\}  _{i}<\beta\left\{  S^{\left(  i\right)
}\right\}  _{i+1}$ and thus%
\[
\gamma_{S^{\prime}}=\prod\limits_{\substack{1\leq a<b\leq2mk\\\beta\left\{
S^{\prime}\right\}  _{a}<\beta\left\{  S^{\prime}\right\}  _{b}}}\left(
1-\frac{1}{\left(  c\left(  a,S^{\prime}\right)  -c\left(  b,S^{\prime
}\right)  \right)  ^{2}}\right)  .
\]
This completes the induction.
\end{proof}

For $w\in\emph{S}_{2mk}$ let $A\left(  w\right)  $ denote the matrix of the
action of $\sigma\left(  w\right)  $ on the basis $\left\{  S:S\in
\mathcal{Y}\left(  \sigma\right)  \right\}  $, so that $\sigma\left(
w\right)  S=\sum_{S^{\prime}}A\left(  w\right)  _{S^{\prime},S}S^{\prime}$.
These matrices are generated by the $A\left(  s_{i}\right)  $ which are
specified in the transformation formulas in Subsection \ref{ieqi1}. The
polynomial $\gamma_{S}J_{\pi\left\{  S\right\}  }$ is a simultaneous
eigenfunction of $\left\{  \omega_{i}\right\}  $ with the same respective
eigenvalues as $S$ and it has the same length, thus it satisfies%
\[
w\gamma_{S}J_{\pi\left\{  S\right\}  }=\sum_{S^{\prime}}A\left(  w\right)
_{S^{\prime},S}\gamma_{S^{\prime}}J_{\pi\left\{  S^{\prime}\right\}  }%
,w\in\mathcal{S}_{2mk},S\in\mathcal{Y}\left(  \sigma\right)  .
\]
Note $w\gamma_{S}J_{\pi\left\{  S\right\}  }\left(  x\right)  =\tau\left(
w\right)  \gamma_{S}J_{\pi\left\{  S\right\}  }\left(  xw\right)  $.

\begin{definition}
The linear map $\mu:\mathcal{P}_{\sigma}\rightarrow\mathcal{P}_{\tau}$ is
given by%
\[
\mu\left(  \sum_{S}f_{S}\left(  x\right)  \otimes S\right)  =\sum_{S}%
f_{S}\left(  x\right)  \gamma_{S}J_{\pi\left\{  S\right\}  }\left(  x\right)
,
\]
where each $f_{S}\in\mathcal{P}$.
\end{definition}

\begin{proposition}
The map $\mu$ commutes with multiplication by $x_{i}$ and with the action of
$\emph{S}_{2mk}$ for $1\leq i\leq2mk$ and $w\in\emph{S}_{2mk}$.
\end{proposition}

\begin{proof}
The first part is obvious. For the second part let $g\left(  x\right)
=f\left(  x\right)  \otimes S$ for some $f\in\mathcal{P}$ and $S\in
\mathcal{Y}\left(  \sigma\right)  $. Then
\begin{align*}
wg\left(  x\right)   &  =f\left(  xw\right)  \otimes\sigma\left(  w\right)
S\\
&  =\sum_{S^{\prime}}A\left(  w\right)  _{S^{\prime},S}f\left(  xw\right)
\otimes S^{\prime},\\
\mu\left(  wg\left(  x\right)  \right)   &  =\sum_{S^{\prime}}A\left(
w\right)  _{S^{\prime},S}f\left(  xw\right)  \gamma_{S^{\prime}}J_{\pi\left\{
S^{\prime}\right\}  }\left(  x\right)  .
\end{align*}
Also%
\begin{align*}
w\left(  \mu(g\left(  x\right)  \right)   &  =f\left(  xw\right)  w\gamma
_{S}J_{\pi\left\{  S\right\}  }\left(  x\right)  \\
&  =f\left(  xw\right)  \sum_{S^{\prime}}A\left(  w\right)  _{S^{\prime}%
,S}\gamma_{S^{\prime}}J_{\pi\left\{  S^{\prime}\right\}  ^{\prime}}\\
&  =\mu\left(  wg\left(  x\right)  \right)  .
\end{align*}

\end{proof}

Recall the key fact: $\mathcal{D}_{i}J_{\beta\left\{  S\right\}  ,T\left\{
S\right\}  }=0$ for $1\leq i\leq2mk$ , at $\kappa_{0}=\frac{1}{m+2}$.

\begin{theorem}
The map $\mu$ commutes with $\mathcal{D}_{i}$ for $1\leq i\leq2mk$.
\end{theorem}

\begin{proof}
Let $g\left(  x\right)  =f\left(  x\right)  \otimes S$ for some $f\in
\mathcal{P}$ and $S\in\mathcal{Y}\left(  \sigma\right)  $ then%
\begin{align*}
\mathcal{D}_{i}g\left(  x\right)   &  =\frac{\partial}{\partial x_{i}}f\left(
x\right)  \otimes S+\kappa_{0}\sum_{j=1,j\neq i}^{2mk}\frac{f\left(  x\right)
-f\left(  x\left(  i,j\right)  \right)  }{x_{i}-x_{j}}\otimes\sigma\left(
\left(  i,j\right)  \right)  S\\
&  =\frac{\partial}{\partial x_{i}}f\left(  x\right)  \otimes S+\kappa_{0}%
\sum_{j=1,j\neq i}^{2mk}\frac{f\left(  x\right)  -f\left(  x\left(
i,j\right)  \right)  }{x_{i}-x_{j}}\otimes\sum_{S^{\prime}}A\left(  \left(
i,j\right)  \right)  _{S^{\prime},S}S^{\prime},\\
\mu\left(  \mathcal{D}_{i}g\left(  x\right)  \right)   &  =\left(
\frac{\partial}{\partial x_{i}}f\left(  x\right)  \right)  \gamma_{S}%
J_{\pi\left\{  S\right\}  }\left(  x\right) \\
&  +\kappa_{0}\sum_{j=1,j\neq i}^{2mk}\frac{f\left(  x\right)  -f\left(
x\left(  i,j\right)  \right)  }{x_{i}-x_{j}}\sum_{S^{\prime}}A\left(  \left(
i,j\right)  \right)  _{S^{\prime},S}\gamma_{S^{\prime}}J_{\pi\left\{
S^{\prime}\right\}  }\left(  x\right)  .
\end{align*}
On the other hand%
\begin{align*}
\mathcal{D}_{i}\left(  \mu g\left(  x\right)  \right)   &  =\left(
\frac{\partial}{\partial x_{i}}f\left(  x\right)  \gamma_{S}J_{\pi\left\{
S\right\}  }\left(  x\right)  \right) \\
&  +\kappa_{0}\sum_{j=1,j\neq i}^{2mk}\tau\left(  \left(  i,j\right)  \right)
\gamma_{S}\frac{f\left(  x\right)  J_{\pi\left\{  S\right\}  }\left(
x\right)  -f\left(  x\left(  i,j\right)  \right)  J_{\pi\left\{  S\right\}
}\left(  x\left(  i,j\right)  \right)  }{x_{i}-x_{j}}\\
&  =\left(  \frac{\partial}{\partial x_{i}}f\left(  x\right)  \right)
\gamma_{S}J_{\pi\left\{  S\right\}  }\left(  x\right)  +f\left(  x\right)
\left(  \frac{\partial}{\partial x_{i}}\gamma_{S}J_{\pi\left\{  S\right\}
}\left(  x\right)  \right) \\
&  +\kappa_{0}f\left(  x\right)  \gamma_{S}\sum_{j=1,j\neq i}^{2mk}\tau\left(
\left(  i,j\right)  \right)  \gamma_{S}\frac{J_{\pi\left\{  S\right\}
}\left(  x\right)  -J_{\pi\left\{  S\right\}  }\left(  x\left(  i,j\right)
\right)  }{x_{i}-x_{j}}\\
&  +\kappa_{0}\sum_{j=1,j\neq i}^{2mk}\tau\left(  \left(  i,j\right)  \right)
\frac{f\left(  x\right)  -f\left(  x\left(  i,j\right)  \right)  }{x_{i}%
-x_{j}}J_{\pi\left\{  S\right\}  }\left(  x\left(  i,j\right)  \right)  .
\end{align*}
But
\begin{align*}
\tau\left(  \left(  i,j\right)  \right)  \gamma_{S}J_{\pi\left\{  S\right\}
}\left(  x\left(  i,j\right)  \right)   &  =\left(  i,j\right)  \gamma
_{S}J_{\pi\left\{  S\right\}  }\left(  x\right) \\
&  =\sum_{S^{\prime}}A\left(  \left(  i,j\right)  \right)  _{S^{\prime}%
,S}\gamma_{S^{\prime}}J_{\pi\left\{  S^{\prime}\right\}  }\left(  x\right)
\end{align*}
and thus $\mathcal{D}_{i}\left(  \mu g\left(  x\right)  \right)  =\mu\left(
\mathcal{D}_{i}g\left(  x\right)  \right)  $. Notice the part $f\left(
x\right)  \mathcal{D}_{i}J_{\pi\left\{  S\right\}  }\left(  x\right)  \ $of
the calculation vanishes.
\end{proof}

The Proposition and the Theorem together show that $\mu$ is a map of modules
of the rational Cherednik algebra (with parameter $\kappa=\frac{1}{m+2}$). In
fact the map can be reversed: define scalar polynomials $p_{S,T}\left(
x\right)  $ by $\gamma_{S}J_{\pi\left\{  S\right\}  }\left(  x\right)
=\sum_{T\in\mathcal{Y}\left(  \tau\right)  }p_{S,T}\left(  x\right)  \otimes
T$ then it can be shown (fairly straightforwardly) that%
\[
q_{T}\left(  x\right)  =\sum_{S}p_{S,T}\left(  x\right)  \frac{\left\Vert
T\right\Vert ^{2}}{\left\Vert S\right\Vert ^{2}}\otimes S
\]
is a singular polynomial in $\mathcal{P}_{\sigma}$ for $\kappa=-\frac{1}{m+2}$
and is of isotype $\tau$. So one can define a map analogous to $\mu$ from
$\mathcal{P}_{\tau}\rightarrow\mathcal{P}_{\sigma}$. There are general results
about duality and maps of modules of the rational Cherednik algebra in
\cite[Sect. 4]{GGOR2003}. In \cite{GGJL2017} there are theorems about the
existence of maps between standard modules in the context of complex
reflection groups.

\section{Further developments and concluding remarks\label{Conclude}}

The construction of singular polynomials in $\mathcal{P}_{\tau}$ which are of
isotype $\sigma$ is easily extendable to $\kappa=\frac{n}{m+2}$ with $n\geq1$
and $\gcd\left(  n,m+2\right)  =1$. Define $\lambda^{\prime}=n\lambda$ then
$J_{\lambda^{\prime},T_{0}}$ is singular for $\kappa=\frac{n}{m+2}$. This is
valid because the uniqueness theorems can be derived from the $n=1$ case:
suppose $\left(  \beta,T\right)  \in\mathbb{N}_{0}^{2mk}\times\mathcal{Y}%
\left(  \tau\right)  $ such that $\beta\trianglelefteq n\lambda$ and
$\frac{m+2}{n}\beta_{i}+c\left(  r_{\beta}\left(  i\right)  ,T\right)
=c\left(  i,S_{0}\right)  $, then $\frac{m+2}{n}\beta_{i}\in\mathbb{N}_{0}$
which implies $\beta_{i}=n\beta_{i}^{\prime}$ , for each $i$. Further
$\beta^{\prime}\trianglelefteq\lambda$ and the uniqueness of $\beta^{\prime}$
shows the same for $\beta$.

It may be possible to extend our analysis to the situation where the top brick
is truncated, that is, $\sigma=\left(  mk+\ell,mk+\ell\right)  $ and
$\tau=\left(  m^{2k},\ell,\ell\right)  $ with $1\leq\ell<m$, but we leave this
for another time.

In this paper we constructed singular nonsymmetric Jack polynomials in
$\mathcal{P}_{\tau}$ which are of isotype $\sigma$. In general, suppose $\tau$
and $\sigma$ are partitions of $N$ and there are singular polynomials in
$\mathcal{P}_{\tau}$ for $\kappa=\kappa_{0}$ which are of isotype $\sigma$,
then it can be shown that there are singular polynomials in $\mathcal{P}%
_{\sigma}$ for $\kappa=-\kappa_{0}$ which are of isotype $\tau$. This idea was
sketched in Section \ref{Maps}. It turns out that interesting new problems may
arise. In the present work we used uniqueness theorems about spectral vectors
to show the validity of specializing NSJP's to $\kappa=\kappa_{0}$ (some fixed
rational) to obtain singular polynomials. However it is possible that some
singular polynomial is a simultaneous eigenfunction of $\left\{
\mathcal{U}_{i}\right\}  $ but is not the specialization of an NSJP.

Our example is for $N=5,\tau=\left(  3,1,1\right)  ,\sigma=\left(  5\right)  $
and $\kappa=\frac{1}{2}$. The singular polynomials for $\mathcal{P}_{\sigma}$
(scalar polynomials) are well-known (\cite{D2004}). In particular $J_{\left(
3,2,0,0,0\right)  }$ has no pole at $\kappa=-\frac{1}{2}$ and is singular
there, furthermore it is of isotype $\tau=\left(  3,1,1\right)  $. From the
general result there are singular polynomials in $\mathcal{P}_{\tau}$ of
isotype $\sigma$ (that is, invariant) for $\kappa=\frac{1}{2}$. The uniqueness
approach fails here. Let%
\begin{align*}
T  &  =%
\begin{bmatrix}
5 & 4 & 3\\
2 &  & \\
1 &  &
\end{bmatrix}
,T^{\prime}=%
\begin{bmatrix}
5 & 3 & 2\\
4 &  & \\
1 &  &
\end{bmatrix}
,\\
\left[  c\left(  \cdot,T\right)  \right]   &  =\left[  -2,-1,2,1,0\right]
,\left[  c\left(  \cdot,T^{\prime}\right)  \right]  =\left[
-2,2,1,-1,0\right]  ,\\
\alpha &  =\left(  3,2,0,0,0\right)  ,\beta=\left(  1,1,2,1,0\right)  .
\end{align*}
The spectral vectors $\left[  2\gamma_{i}+c\left(  r_{\gamma}\left(  i\right)
,T\right)  \right]  _{i=1}^{5}$ are (note $r_{\beta}=\left(  2,3,1,4,5\right)
$)%
\begin{align*}
\zeta_{\alpha,T}^{\prime}  &  =\left(  6-2,4-1,2,1,0\right)  =\left(
4,3,2,1,0\right)  ,\\
\zeta_{\beta,T^{\prime}}^{\prime}  &  =\left(  2+2,2+1,4-2,2-1,0\right)
=\left(  4,3,2,1,0\right)  .
\end{align*}
By direct (symbolic computation assisted) calculation we find that both
$J_{\alpha,T}$ and $J_{\beta,T^{\prime}}$ are defined (no pole) at
$\kappa=\frac{1}{2}$, neither is singular or of isotype $\left(  5\right)  $
(invariant under each $s_{i}$) but%
\[
\mathcal{D}_{i}\left(  J_{\alpha,T}+2J_{\beta,T^{\prime}}\right)  =0,~1\leq
i\leq5,~\kappa=\frac{1}{2}.
\]
Also $J_{\alpha,T}+2J_{\beta,T^{\prime}}$ is invariant and $\mathcal{U}%
_{i}^{\prime}\left(  J_{\alpha,T}+2J_{\beta,T^{\prime}}\right)  =\left(
5-i\right)  \left(  J_{\alpha,T}+2J_{\beta,T^{\prime}}\right)  $ for $1\leq
i\leq5$. The polynomial $J_{\alpha,T}$ is a sum of 100 monomials in $x$, with
coefficients in $V_{\tau}$.

We suspect that our results benefitted from the fact that $\tau$ and $\sigma$
are rectangular partitions, and that the analysis of singular polynomials for
other partitions (hook tableaux for example) becomes significantly more difficult.

\end{document}